\documentclass[11pt]{article}
\usepackage{amsmath}
\usepackage{amssymb}
\usepackage{amsthm}
\usepackage{enumerate}
\usepackage[colorlinks]{hyperref}
\usepackage{enumitem}
\usepackage[margin=1in]{geometry}
\usepackage{stmaryrd}
\usepackage{pgfplots, pgfplotstable, booktabs, colortbl, array, multirow}

\usepackage{tikz}

\usepackage[style=numeric-comp,giveninits=true,url=false,maxbibnames=5,backend=bibtex]{biblatex}
\bibliography{references.bib}

\theoremstyle{plain}
\newtheorem{theorem}{Theorem}[section]
\newtheorem{lemma}[theorem]{Lemma}
\newtheorem{proposition}[theorem]{Proposition}
\newtheorem{corollary}[theorem]{Corollary}

\theoremstyle{definition}
\newtheorem{definition}[theorem]{Definition}

\newtheorem{remark}[theorem]{Remark}

\DeclareMathOperator{\dv}{div}
\DeclareMathOperator{\sym}{sym}
\DeclareMathOperator{\Tr}{Tr}
\DeclareMathOperator{\Ric}{Ric}
\DeclareMathOperator{\sff}{\mathrm{I\!I}}
\DeclareMathOperator{\Riem}{Riem}

\DeclareMathOperator{\ein}{ein}
\DeclareMathOperator{\df}{def}

\newcommand{\gt}{\widetilde{g}}

\newcommand{\nrm}[1]{\left|\!\left|\!\left|{#1}\right|\!\right|\!\right|}
\newcommand{\nrmt}[1]{\left|\!\left|\!\left|\!\left|{#1}\right|\!\right|\!\right|\!\right|}

\newcommand{\textoverline}[1]{$\overline{\mbox{#1}}$}

\makeatletter
\DeclareFontFamily{OMX}{MnSymbolE}{}
\DeclareSymbolFont{MnLargeSymbols}{OMX}{MnSymbolE}{m}{n}
\SetSymbolFont{MnLargeSymbols}{bold}{OMX}{MnSymbolE}{b}{n}
\DeclareFontShape{OMX}{MnSymbolE}{m}{n}{
    <-6>  MnSymbolE5
   <6-7>  MnSymbolE6
   <7-8>  MnSymbolE7
   <8-9>  MnSymbolE8
   <9-10> MnSymbolE9
  <10-12> MnSymbolE10
  <12->   MnSymbolE12
}{}
\DeclareFontShape{OMX}{MnSymbolE}{b}{n}{
    <-6>  MnSymbolE-Bold5
   <6-7>  MnSymbolE-Bold6
   <7-8>  MnSymbolE-Bold7
   <8-9>  MnSymbolE-Bold8
   <9-10> MnSymbolE-Bold9
  <10-12> MnSymbolE-Bold10
  <12->   MnSymbolE-Bold12
}{}

\let\llangle\@undefined
\let\rrangle\@undefined
\DeclareMathDelimiter{\llangle}{\mathopen}%
                     {MnLargeSymbols}{'164}{MnLargeSymbols}{'164}
\DeclareMathDelimiter{\rrangle}{\mathclose}%
                     {MnLargeSymbols}{'171}{MnLargeSymbols}{'171}
\makeatother

\begin{document}

\title{Finite element approximation of the Einstein tensor}

\author{Evan S. Gawlik\thanks{Department of Mathematics, University of Hawai`i at M\textoverline{a}noa, Honolulu, HI, 96822, USA, \href{egawlik@hawaii.edu}{egawlik@hawaii.edu} } \and Michael Neunteufel\thanks{Institute for Analysis and Scientific Computing, TU Wien, Wiedner Hauptstr.~8-10, 1040 Wien, Austria, \href{michael.neunteufel@tuwien.ac.at}{michael.neunteufel@tuwien.ac.at}}}

\date{}

\maketitle

\begin{abstract}
We construct and analyze finite element approximations of the Einstein tensor in dimension $N \ge 3$.  We focus on the setting where a smooth Riemannian metric tensor $g$ on a polyhedral domain $\Omega \subset \mathbb{R}^N$ has been approximated by a piecewise polynomial metric $g_h$ on a simplicial triangulation $\mathcal{T}$ of $\Omega$ having maximum element diameter $h$.  We assume that $g_h$ possesses single-valued tangential-tangential components on every codimension-1 simplex in $\mathcal{T}$.  Such a metric is not classically differentiable in general, but it turns out that one can still attribute meaning to its Einstein curvature in a distributional sense.  We study the convergence of the distributional Einstein curvature of $g_h$ to the Einstein curvature of $g$ under refinement of the triangulation.  We show that in the $H^{-2}(\Omega)$-norm, this convergence takes place at a rate of $O(h^{r+1})$ when $g_h$ is an optimal-order interpolant of $g$ that is piecewise polynomial of degree $r \ge 1$.  We provide numerical evidence to support this claim.
\end{abstract}

\section{Introduction} \label{sec:intro}

The Einstein tensor $G = \Ric - \frac{1}{2}Rg$ encodes important information about the geometry of a pseudo-Riemannian manifold.  It is a symmetric $(0,2)$-tensor field constructed from the Ricci tensor $\Ric$, the scalar curvature $R$, and the metric tensor $g$ whose divergence vanishes by the contracted Bianchi identity.  It features prominently in the Einstein field equations of general relativity, which relate $G$ to the stress-energy permeating spacetime.  

This paper is devoted to the design and analysis of finite element approximations of the Einstein tensor in dimension $N \ge 3$.  (In dimension $N=2$, the Einstein tensor automatically vanishes.)  We consider the setting where a smooth Riemannian metric tensor on a polyhedral domain $\Omega \subset \mathbb{R}^N$ is approximated by a metric tensor belonging to the \emph{Regge finite element space}: the space of symmetric $(0,2)$-tensor fields on $\Omega$ that are piecewise polynomial with respect to a simplicial triangulation $\mathcal{T}$ of $\Omega$ and possess single-valued tangential-tangential components on every codimension-1 simplex in $\mathcal{T}$ \cite{christiansen2004characterization, christiansen2011linearization, li2018regge}.  In general, we refer to any piecewise smooth (not necessarily piecewise polynomial) Riemannian metric tensor with the aforementioned continuity property as a \emph{Regge metric}.  We address two questions: (1) Given a Regge metric, how should one ascribe meaning to its Einstein tensor?  (2) If that Regge metric is piecewise polynomial and approximates a smooth Riemannian metric in a suitable sense, is its Einstein tensor close to the Einstein tensor of the smooth metric?  We focus on the case where the metrics have positive signature for simplicity, but it is conceivable that one could generalize our analysis to metrics with indefinite signature.

Our main result, Theorem~\ref{thm:conv}, states that if $g$ is a smooth Riemannian metric on $\Omega$ and if $g_h$ is a Regge metric on $\mathcal{T}$ that interpolates $g$ and is piecewise polynomial of degree $r \ge 1$, then the Einstein curvature of $g_h$ (as defined in Definition~\ref{def:distein}) differs from that of $g$ by $O(h^{r+1})$ in the $H^{-2}(\Omega)$-norm.  Here $h$ denotes the maximum diameter of the $N$-simplices in $\mathcal{T}$.  Our numerical experiments indicate that this estimate is sharp.  In particular, the $H^{-2}(\Omega)$-error in the Einstein tensor is generally $O(1)$ when $r=0$, not $O(h)$.  These results are consistent with the results in~\cite{gawlik2023scalar}, where analogous error estimates were proved for the scalar curvature.

Below we discuss our definition of the Einstein tensor for Regge metrics and our strategy for proving error estimates.  To simplify the discussion, we consider the case where $\Omega$ is a compact domain with no boundary for the moment. Concretely, we suppose $\Omega$ is an $N$-dimensional cube with opposing faces identified.  The treatment of bounded domains will begin in Section~\ref{sec:curvature}.

\paragraph{Defining the Einstein tensor for Regge metrics.}

The question of how to define the Einstein tensor for Regge metrics is nontrivial because Regge metrics are not classically differentiable.  
Thus, the standard coordinate formulas for $\Ric$ and $R$ involving zeroth, first, and second derivatives of the metric are inapplicable.  Nevertheless, prior work suggests that Regge metrics do admit a natural notion of scalar curvature in a distributional sense~\cite{berchenko2022finite,strichartz2020defining,gawlik2023scalar}.  Given a Regge metric $g$ with volume form $\omega$, one defines the \emph{distributional (densitized) scalar curvature} of $g$ to be a linear functional $(R\omega)_{\rm dist}(g)$ whose action on any scalar field $v$ (possessing suitable regularity) is defined by
\begin{equation} \label{distcurv_intro}
	\llangle (R\omega)_{\rm dist}(g), v \rrangle = \sum_{T \in \mathcal{T}^N} \int_T R_T v\, \omega_T + 2 \sum_{F \in \mathcal{T}^{N-1}} \int_F \llbracket H \rrbracket_F v\, \omega_F + 2 \sum_{S \in \mathcal{T}^{N-2}} \int_S \Theta_S v \,\omega_S.
\end{equation}
In this formula, $\mathcal{T}^k$ denotes the set of all $k$-simplices in $\mathcal{T}$, $\omega_D$ denotes the induced volume form on a simplex $D$, $R_T$ denotes the (classically defined) scalar curvature of $g|_T$, $\llbracket H \rrbracket_F$ denotes the jump in the mean curvature $H$ across $F$, and $\Theta_S$ denotes the \emph{angle defect} along $S$: $2\pi$ minus the sum of the dihedral angles between pairs of adjacent $(N-1)$-simplices emanating from $S$, all measured with respect to $g$.  This definition generalizes the definition $\llangle (R\omega)_{\rm dist}(g), v \rrangle = 2 \sum_{S \in \mathcal{T}^{N-2}} \int_S \Theta_S v \omega_S$ posited for piecewise constant Regge metrics in~\cite{regge1961general,cheeger1984curvature,christiansen2013exact}.  (The factor of two reflects the fact that the angle defect approximates Gaussian curvature in two dimensions, which is half the scalar curvature $R$.)

It was observed in~\cite{gawlik2023scalar} that by taking variations of~\eqref{distcurv_intro}, one is led to a natural way of defining the Einstein tensor for Regge metrics.   Namely, consider a family of Regge metrics $g(t)$ with time derivative $\sigma = \frac{\partial}{\partial t} g$.  In~\cite[Theorem 3.6]{gawlik2023scalar} we showed that
\begin{equation} \label{distcurvdot}
	\frac{d}{dt} \llangle (R\omega)_{\rm dist}(g(t)), v \rrangle = b_h(g;\sigma,v) - a_h(g;\sigma,v)
\end{equation}
for all $v$, where $b_h(g;\cdot,\cdot)$ and $a_h(g;\cdot,\cdot)$ are certain mesh-dependent bilinear forms, the first of which is manifestly a distributional version of $\int_\Omega (\dv\dv\mathbb{S} \sigma) v \, \omega$. Here $\mathbb{S}\sigma = \sigma - g\Tr\sigma$, $\dv$ denotes the covariant divergence operator, and $\Tr\sigma$ is the trace of $\sigma$ with respect to $g$.  What stands out about~\eqref{distcurvdot} is its resemblance to the corresponding formula that holds for a family of smooth Riemannian metrics $g(t)$ with time derivative $\sigma = \frac{\partial}{\partial t} g$:
\begin{equation} \label{curvdot}
	\frac{d}{dt} \int_\Omega R v \,\omega =  \int_\Omega (\dv\dv\mathbb{S} \sigma) v \, \omega - \int_\Omega \langle G, \sigma \rangle v\, \omega.
\end{equation}
Here, $\langle \cdot, \cdot \rangle$ denotes the $g$-inner product of tensor fields.

Inspired by the correspondence between~\eqref{distcurvdot} and~\eqref{curvdot}, we posited the following definition in~\cite[Section 3.2]{gawlik2023scalar}.  Given a Regge metric $g$, the \emph{distributional (densitized) Einstein tensor} associated with $g$ is the linear functional $(G\omega)_{\rm dist}(g)$ whose action on any symmetric $(0,2)$-tensor field $\sigma$ (possessing suitable regularity) is defined by
\[
\llangle (G\omega)_{\rm dist}(g), \sigma \rrangle = a_h(g;\sigma,1) = -\left.\frac{d}{dt}\right|_{t=0} \llangle (R\omega)_{\rm dist}(g+t\sigma), 1 \rrangle.
\]
Equivalently, upon recalling the expression for $a_h$ derived in~\cite{gawlik2023scalar}, this definition reads 
\begin{equation} \label{distein}
	\llangle (G\omega)_{\rm dist}(g), \sigma \rrangle = \sum_{T \in \mathcal{T}^N} \int_T \langle G_T, \sigma \rangle\, \omega_T + \sum_{F \in \mathcal{T}^{N-1}} \int_F \langle \llbracket \overline{\sff} \rrbracket_F, \sigma|_F \rangle\, \omega_F - \sum_{S \in \mathcal{T}^{N-2}} \int_S \langle \Theta_S g|_S, \sigma|_S \rangle \,\omega_S.
\end{equation}
Here, $G_T$ denotes the (classically defined) Einstein tensor associated with $g|_T$, $\cdot|_D$ denotes the pullback under the inclusion $D \hookrightarrow \Omega$ of a simplex $D$, and $\llbracket \overline{\sff}\rrbracket_F$ denotes the jump in the \emph{trace-reversed second fundamental form} $\overline{\sff}$ across $F$; see Sections~\ref{sec:evolution} and \ref{sec:curvature} for more notational details.

The appearance of $\llbracket \overline{\sff} \rrbracket_F$ in~\eqref{distein} is noteworthy.  Precisely the same quantity arises in general relativity when treating stress-energy sources that are concentrated on a spacetime hypersurface $F$.  The interface conditions that arise in that setting relate $\llbracket \overline{\sff} \rrbracket_F$ to the distribution of stress-energy on $F$ and are known in the physics literature as the \emph{Israel junction conditions}~\cite{israel1966singular}.

\paragraph{Strategy behind the analysis.}

Our strategy for proving error estimates for the Einstein tensor mirrors the strategy we used in~\cite{gawlik2023scalar} to prove error estimates for the scalar curvature.  Namely, given a piecewise polynomial Regge metric $g_h$ that approximates a smooth Riemannian metric $g$, we consider an evolving metric $\widetilde{g}(t) = (1-t)g+tg_h$ with time derivative $\sigma = \frac{\partial}{\partial t} \widetilde{g} = g_h-g$ and derive an integral formula for the error which reads
\[
\llangle (G\omega)_{\rm dist}(g_h) - (G\omega)(g), \rho \rrangle = \int_0^1 \left( B_h(\widetilde{g}(t); \sigma, \rho) + A_h(\widetilde{g}(t); \sigma, \rho) \right) \, dt.
\]
Here $\rho$ is an arbitrary symmetric $(0,2)$-tensor field possessing sufficient regularity, and $B_h(\widetilde{g};\cdot,\cdot)$ and $A_h(\widetilde{g};\cdot,\cdot)$ are certain mesh-dependent bilinear forms.  We then prove upper bounds for $B_h$ and $A_h$, yielding an error estimate for $(G\omega)_{\rm dist}(g_h) - (G\omega)(g)$ in a negative-order Sobolev norm.

\paragraph{Other comments.}
It is worth commenting on the bilinear forms $B_h(\widetilde{g};\cdot,\cdot)$ and $A_h(\widetilde{g};\cdot,\cdot)$, as they are interesting in their own right.  Their sum is the second variation of $-\llangle (R\omega)_{\rm dist}(g), 1 \rrangle$ around the metric $g=\widetilde{g}$, and, consequently, $B_h(\widetilde{g};\cdot,\cdot) + A_h(\widetilde{g};\cdot,\cdot)$ is a symmetric bilinear form.  When $\widetilde{g}$ is the Euclidean metric, $A_h(\widetilde{g};\cdot,\cdot)$ vanishes, and $B_h(\widetilde{g};\sigma,\rho)$ coincides with $\llangle \ein_{\rm dist} \sigma, \rho \rrangle$ for all smooth $\rho$ with compact support.  Here, $\ein$ denotes the second-order, linear differential operator obtained from linearizing the map $g \mapsto G(g)$ around the Euclidean metric,\footnote{We also use $\ein$ more generally to denote (the principal part of) the linearization of $g \mapsto G(g)$ around a non-Euclidean metric in the sequel; see~\eqref{linein} and~\eqref{Gdot_general}.} and $\ein_{\rm dist} \sigma$ denotes the action of $\ein$, interpreted in a distributional sense, on a piecewise smooth symmetric $(0,2)$-tensor field $\sigma$ possessing tangential-tangential continuity.  This correspondence between $\ein_{\rm dist}$ and the second variation of $-\llangle (R\omega)_{\rm dist}(g), 1 \rrangle$ around the Euclidean metric was first discovered for piecewise constant $\sigma$ by Christiansen~\cite{christiansen2011linearization}.

There is a sense in which the bilinear forms $a_h$ and $b_h$ in~\eqref{curvdot}, as well as the distributional scalar curvature $(R\omega)_{\rm dist}$, are ``traces'' of $A_h$, $B_h$, and $(G\omega)_{\rm dist}$.  More precisely, for all Regge metrics $g$, all symmetric $(0,2)$-tensor fields $\sigma$ possessing tangential-tangential continuity, and all smooth functions $v$,
\begin{align*}
	\llangle (G\omega)_{\rm dist}(g), vg \rrangle &= -\left(\frac{N-2}{2}\right) \llangle (R\omega)_{\rm dist}(g), v \rrangle, \\
	B_h(g; \sigma, vg) &= -\left( \frac{N-2}{2}\right) b_h(g; \sigma,v), \\
	A_h(g; \sigma, vg) &= \left( \frac{N-4}{2} \right) a_h(g; \sigma, v).
\end{align*} 
See Section~\ref{sec:rel_scalar_curv} for details.  A consequence of this correspondence is that many of our calculations and results below recover ones in~\cite{gawlik2023scalar} upon ``taking traces''.

\paragraph{Structure of the paper.} In Section~\ref{sec:evolution} we discuss the evolution of geometric quantities such as the Einstein tensor and the trace-reversed second fundamental form under a changing metric. The distributional densitized Einstein tensor and its evolution are presented in Section~\ref{sec:curvature}. There we also discuss connections with the distributional densitized scalar curvature and the distributional covariant linearized Einstein operator. In Section~\ref{sec:convergence} the analysis of the distributional Einstein tensor is performed by estimating the bilinear forms that characterize its evolution. A numerical example is presented in Section~\ref{sec:numerical} showing that the derived error estimates are sharp.

\section{Evolution of geometric quantities} \label{sec:evolution}

In this section, we study the evolution of various geometric quantities under metric deformations. 

We adopt the following notation.  Let $M$ be an $N$-dimensional manifold equipped with a smooth Riemannian metric $g$.  The Levi-Civita connection associated with $g$ is denoted $\nabla$.  If $\sigma$ is a $(p,q)$-tensor field, then its covariant derivative is the $(p,q+1)$-tensor field $\nabla\sigma$, and its covariant derivative in the direction of a vector field $X$ is the $(p,q)$-tensor field $\nabla_X \sigma$.  Its trace $\Tr\sigma$ is the contraction of $\sigma$ along the first two indices, using $g$ to raise or lower indices as needed.  We denote $\dv\sigma = \Tr\nabla\sigma$ and $\Delta \sigma = \dv \nabla \sigma$.  The pointwise $g$-inner product of two $(p,q)$-tensor fields $\sigma$ and $\rho$ is denoted $\langle \sigma, \rho \rangle$.  For vector fields $X$ and $Y$, we often write $g(X,Y)$ instead of $\langle X, Y \rangle$. When we wish to emphasize the dependence of $\nabla$, $\nabla_X$, $\dv$, $\langle \cdot, \cdot \rangle$, etc.~on $g$, we write $\nabla_g$, $\nabla_{g,X}$, $\dv_g$, $\langle \cdot, \cdot \rangle_g$, etc.

The volume form associated with $g$ is denoted $\omega$.  The Riemann curvature tensor, Einstein tensor, Ricci tensor, and scalar curvature of $g$ are denoted $\Riem$, $G$, $\Ric$, and $R$, respectively.  When we wish to emphasize their dependence on $g$, we write $\omega(g)$, $\Riem(g)$, $G(g)$, $\Ric(g)$, and $R(g)$.

If $D$ is an embedded submanifold of $M$, then we denote by $\omega_D$ the induced volume form on $D$.  If $\sigma$ is a tensor field on $M$, then $\left.\sigma\right|_D$ denotes the pullback of $\sigma$ under the inclusion $D \hookrightarrow M$.  Later we will introduce some additional notation related to embedded submanifolds of codimension 1, like the mean curvature $H$, second fundamental form $\sff$, and trace-reversed second fundamental form $\overline{\sff}$; see Section~\ref{sec:evolution_fundamentalform}.

We denote the exterior derivative of a differential form $\alpha$ by $d\alpha$.  If $\alpha$ is a one-form, then $\alpha^\sharp$ denotes the vector field obtained by raising indices with $g$.   If $X$ is a vector field, then $X^\flat$ denotes the one-form $g(X,\cdot)$ obtained by lowering indices with $g$.  If $f$ is a scalar field, then we sometimes interpret the one-form $\nabla f=df$ as the vector field $(df)^\sharp$ without explicitly writing it.  The symmetric covariant derivative of a one-form $\alpha$ is denoted $\df \alpha$.  That is, $(\df \alpha)(X,Y) = \frac{1}{2} ( (\nabla_X \alpha)(Y) + (\nabla_Y \alpha)(X))$ for all vectors $X,Y$.

Occasionally we perform calculations in coordinates.  We adopt the Einstein summation convention throughout.  We use $R_{ijkl}$ and $R_{ij}$ to denote the components of $\Riem$ and $\Ric$, respectively.  Our convention is that 
\begin{equation} \label{Rijkl}
	g((\nabla_Y \nabla_X - \nabla_X \nabla_Y + \nabla_{[X,Y]})Z,W) = R_{ijk\ell} X^i Y^j Z^k W^\ell
\end{equation}
and 
\[
R_{ij} = R^k_{\phantom{k}ikj}.
\]

\subsection{Evolution of the densitized Einstein tensor}

First we study the evolution of the densitized Einstein tensor $G\omega$ under deformations of the metric.

\begin{proposition} \label{prop:eindot}
	Let $g(t)$ be a family of smooth Riemannian metrics with time derivative $\frac{\partial}{\partial t}g =: \sigma$.  Let $\rho$ be an arbitrary time-independent symmetric $(0,2)$-tensor field. We have
	\begin{align}
		2\frac{\partial}{\partial t}& (\langle G,\rho\rangle\,\omega) = 2\left( \langle \dot{G}, \rho \rangle + \frac{1}{2} \langle G, \rho \rangle \Tr \sigma  - 2\sigma:G:\rho \right) \omega \nonumber \\
		&=\left(\langle 2\ein \sigma,\rho\rangle + 2\sigma:\Riem:\rho + \langle\Ric,\sigma\rangle \Tr\rho + \langle\Ric,\rho\rangle\Tr\sigma + R\langle J\sigma,\rho\rangle - 2\sigma:\Ric:\rho\right) \omega, \label{eindot}
	\end{align}
	where $\sigma:\Ric:\rho=\sigma^i_{\phantom{i}j}R^j_{\phantom{j}k}\rho^k_{\phantom{k}i}$, $\sigma:\Riem:\rho=R^{ki}_{\phantom{ki}j\ell}\sigma^\ell_{\;k}\rho^j_{\phantom{j}i}$,
	\[
	J \sigma = \sigma - \frac{1}{2}g \Tr \sigma,
	\]
	and
	\begin{equation} \label{linein}
		\ein = J\df\dv J - \frac{1}{2} \Delta J.
	\end{equation}
\end{proposition}
\begin{remark}
	We will refer to the operator $\ein$ as the \emph{covariant linearized Einstein operator}, or simply the linearized Einstein operator.  It is a second-order, linear differential operator that sends symmetric $(0,2)$-tensor fields to symmetric $(0,2)$-tensor fields.  It is self-adjoint with respect to the $L^2(\Omega,g)$-inner product.  One can verify this by first checking that $J$ commutes with $\Delta$ and is self-adjoint in a pointwise sense: $\langle J \sigma, \rho \rangle = \langle \sigma, J \rho \rangle$ for all $\sigma$ and $\rho$.
\end{remark}

We will prove Proposition~\ref{prop:eindot} after first studying the evolution of $G$.  For a time-dependent metric $g(t)$ with time derivative $\sigma = \dot{g}$, we have for the evolution of the Einstein tensor
\begin{align*}
	\dot{G} 
	&= \dot{\Ric} - \frac{1}{2} \dot{R} g - \frac{1}{2} R \sigma.
\end{align*}
Equation (2.31) from~\cite{chow2006hamilton} tells us that
\[
\dot{\Ric} = -\frac{1}{2}( \Delta_L \sigma + \nabla \nabla \Tr \sigma - 2\df\dv\sigma),
\]
where $\Delta_L$ is the Lichnerowicz Laplacian, which is given in coordinates by
\[
(\Delta_L \sigma)_{ij} = (\Delta \sigma)_{ij} - 2 R^k_{\;ij\ell} \sigma^\ell_{\;k} - R_{ik}\sigma^k_{\;j} - R_{jk} \sigma^k_{\;i}.
\]
(Note that our sign convention in~\eqref{Rijkl} differs from~\cite{chow2006hamilton}'s.)  We can use the identity $\nabla \nabla \Tr\sigma = \df \nabla \Tr\sigma = \df\dv(g\Tr\sigma)$ to write this as
\[
\dot{\Ric} = -\frac{1}{2}( \Delta_L \sigma - 2\df\dv J\sigma).
\]
Meanwhile, the identity~\cite[Equation (2.30)]{chow2006hamilton}
\[
\dot{R} = \dv\dv \sigma - \Delta \Tr \sigma - \langle \Ric, \sigma \rangle 
\]
implies that with $\mathbb{S}\sigma := \sigma-g\Tr\sigma$,
\[
\dot{R} = \dv\dv\mathbb{S}\sigma - \langle \Ric, \sigma \rangle.
\]
Thus,
\begin{align*}
	\dot{G}
	&= -\frac{1}{2}( \Delta_L \sigma - 2\df\dv J\sigma ) - \frac{1}{2} (\dv\dv\mathbb{S}\sigma - \langle \Ric, \sigma \rangle) g - \frac{1}{2} R \sigma.
\end{align*}
Now we note that since $\nabla g = 0$ and $\Tr$ commutes with covariant differentiation,
\begin{align*}
	\Tr(\df\dv J\sigma) 
	&= \dv\dv\mathbb{S}\sigma + \frac{1}{2}\dv\dv(g\Tr\sigma) = \dv\dv\mathbb{S}\sigma + \frac{1}{2}\Tr\Delta\sigma 
\end{align*}
and
\begin{align*}
	2J\df\dv J\sigma 
	&= 2\df\dv J \sigma - g\Tr(\df\dv J\sigma) = 2\df\dv J\sigma - g\dv\dv\mathbb{S}\sigma - \frac{1}{2} g \Tr\Delta\sigma.
\end{align*}
Thus,
\begin{align*}
	2\dot{G} 
	&= -\Delta_L \sigma + 2 \df \dv J \sigma - g\dv\dv\mathbb{S}\sigma + \langle \Ric, \sigma \rangle g - R\sigma \\
	&= -\Delta_L \sigma + 2 J \df \dv J \sigma + \frac{1}{2}g\Tr\Delta \sigma + \langle \Ric, \sigma \rangle g - R\sigma.
\end{align*}

\begin{lemma}\label{lem:lichn_lapl_J}
	We have
	\[
	J\Delta_L \sigma = \Delta_L J\sigma =\Delta_L \sigma - \frac{1}{2} g \Tr \Delta \sigma.
	\]
\end{lemma}
\begin{proof}
	We can relate the trace of $\Delta_L\sigma$ to the trace of $\Delta \sigma$ by computing
	\begin{align*}
		\Tr\Delta_L\sigma 
		&= g^{ij} (\Delta_L \sigma)_{ij} \\
		&= g^{ij} (\Delta \sigma)_{ij} - 2 g^{ij} R^k_{\;ij\ell} \sigma^\ell_{\;k} - g^{ij} R_{ik}\sigma^k_{\;j} - g^{ij} R_{jk} \sigma^k_{\;i} \\
		&= \Tr\Delta \sigma + 2 R^k_{\;\ell} \sigma^{\ell}_{\;k} - R^j_{\;k} \sigma^k_{\;j} - R^i_{\;k} \sigma^k_{\;i} \\
		&= \Tr\Delta\sigma.
	\end{align*}
	Above, we used the fact that the trace of $\Riem$ with respect to any two indices is always either $\Ric$, $-\Ric$, or zero; in particular it is $-\Ric$ when we take the trace with respect to the middle two indices.  
	It follows that
	\[
	J\Delta_L \sigma = \Delta_L \sigma - \frac{1}{2} g \Tr \Delta \sigma.
	\]
	On the other hand,
	\begin{align*}
		(\Delta_L(g\Tr\sigma))_{ij} 
		&= (\Delta(g\Tr\sigma))_{ij} - 2 R^k_{\;ij\ell} g^\ell_{\;k} \Tr\sigma - R_{ik} g^k_{\;j} \Tr\sigma - R_{jk} g^k_{\;i} \Tr\sigma \\
		&= g_{ij} \Delta \Tr\sigma +2 R_{ij} \Tr \sigma - R_{ij} \Tr \sigma - R_{ji} \Tr \sigma \\
		&= g_{ij} \Delta \Tr\sigma \\
		&= g_{ij} \Tr \Delta \sigma.
	\end{align*}
	This implies that
	\[
	\Delta_LJ \sigma = J\Delta_L \sigma.
	\]
\end{proof}

It follows from Lemma~\ref{lem:lichn_lapl_J} that with $2 \ein := 2J\df\dv J - \Delta J$, we have
\begin{align}
	2\dot{G} 
	&= 2J\df\dv J \sigma - \Delta_LJ \sigma  + \langle \Ric, \sigma \rangle g - R\sigma \nonumber \\
	&= 2\ein\sigma+2R^k_{\;ij\ell}\sigma^\ell_{\;k}+R_{ik}\sigma^{k}_{\;j}+R_{jk}\sigma^k_{\;i} + \langle \Ric, \sigma \rangle g - R\sigma. \label{Gdot_general}
\end{align}

\begin{proof}[Proof of Proposition~\ref{prop:eindot}]
	The first equality in~(\ref{eindot}) is the product rule, where the second term comes from differentiating $\omega$ and the third term comes from derivatives of $g$ when differentiating the inner product $\langle G, \rho \rangle = g^{ij} G_{jk} g^{k\ell} \rho_{\ell i}$.
	
	We compute the second equality in~(\ref{eindot}) by using \eqref{Gdot_general}, $G=\Ric-\frac{1}{2}Rg$, and $\sigma:g:\rho = \langle \sigma,\rho\rangle$ to obtain
	\begin{align*}
		&2\frac{d}{dt}  \left(\langle G, \rho \rangle\, \omega\right) \\
		&=  \left( \langle 2\ein\sigma, \rho \rangle + 2\sigma : \Riem:\rho + 2\sigma:\Ric:\rho + \langle \Ric, \sigma \rangle \Tr \rho  - R \langle \sigma, \rho \rangle + \langle G, \rho \rangle \Tr \sigma - 4\sigma:G:\rho \right) \omega \\
		&=  \big( \langle 2\ein\sigma, \rho \rangle + 2\sigma : \Riem:\rho + 2\sigma:\Ric:\rho + \langle \Ric, \sigma \rangle \Tr \rho  - R \langle \sigma, \rho \rangle + \langle \Ric, \rho \rangle \Tr \sigma-\frac{1}{2}R\Tr\sigma\Tr\rho \\&\quad\quad - 4\sigma:G:\rho \big) \omega \\
		&=  \big( \langle 2\ein\sigma, \rho \rangle + 2\sigma : \Riem:\rho + 2\sigma:\Ric:\rho + \langle \Ric, \sigma \rangle \Tr \rho  + R \langle J\sigma, \rho \rangle + \langle \Ric, \rho \rangle \Tr \sigma \\&\quad\quad - 4\sigma:\Ric:\rho \big) \omega \\
		&=  \left( \langle 2\ein\sigma, \rho \rangle + 2\sigma : \Riem:\rho + \langle \Ric, \sigma \rangle \Tr \rho + \langle \Ric, \rho \rangle \Tr \sigma  + R \langle J\sigma, \rho \rangle - 2\sigma:\Ric:\rho \right) \omega.
	\end{align*}
\end{proof}

\subsection{Evolution of the trace-reversed second fundamental form} \label{sec:evolution_fundamentalform}

Next we study the evolution of the trace-reversed second fundamental form $\overline{\sff}$ on a hypersurface $F$ with unit normal $n$. 

We use the notation
\[
\sff(X,Y) = g(\nabla_X n, Y) = -g(n, \nabla_X Y)
\]
for the second fundamental form on $F$ and $H = \Tr \sff$ for the mean curvature of $F$.  Our sign convention is such that $H$ is positive for a sphere with an outward normal vector.  The \emph{trace-reversed second fundamental form} is $\overline{\sff} = \sff - H g|_F$; that is,
\[
\overline{\sff}(X,Y) = \sff(X,Y) - H g(X,Y)
\]
for all vectors $X,Y$ tangent to $F$.  

We also let $\nabla_F$ and $\dv_F$ denote the surface gradient and surface divergence operators on $F$, which have the following meanings.  
For any $(0,q)$-tensor field $\sigma$ defined on the ambient manifold $M$, $\nabla_F\sigma$ is the $(0,q+1)$-tensor field defined at points that lie on $F$ by
\[
\nabla_F \sigma= \nabla \sigma -n^\flat\otimes \nabla_{n}\sigma.
\]
That is,
\[
(\nabla_F \sigma)(X_1,X_2,\dots,X_{q+1}) = (\nabla_{X_1} \sigma)(X_2,\dots,X_{q+1}) - g(n,X_1) (\nabla_{n}\sigma)(X_2,\dots,X_{q+1})
\]
for all vectors $X_1,X_2,\dots,X_{q+1}$ (not necessarily tangent to $F$).  Likewise, $\dv_F \sigma$ is the $(0,q-1)$-tensor field defined at points that lie on $F$ by
\[
(\dv_F \sigma)(X_1,X_2,\dots,X_{q-1}) = (\dv \sigma)(X_1,X_2,\dots,X_{q-1}) - (\nabla_{n}\sigma)(n,X_1,X_2,\dots,X_{q-1})
\]
for all vectors $X_1,X_2,\dots,X_{q-1}$ (not necessarily tangent to $F$).  Note that $(\dv_F \sigma)|_F = \Tr((\nabla_F \sigma)|_F)$.  Often we will abuse notation and write $(\nabla_F \sigma)(X_1,\dots,X_q)$ as shorthand for $(\nabla_F \sigma)(\cdot,X_1,\dots,X_q)$.  For example, when $\sigma$ is a $(0,2)$-tensor field, $(\nabla_F \sigma)(X_1,X_2)$ is shorthand for the one-form $(\nabla_F \sigma)(\cdot,X_1,X_2)$.  Similarly, $(\nabla_F \sigma)(X_1,\cdot)$ is shorthand for the $(0,2)$-tensor field $(\nabla_F \sigma)(\cdot,X_1,\cdot)$.  We do the same with the operator $\nabla$.

Recall that the surface divergence operator satisfies the identity
\begin{equation} \label{surfacestokes}
	\int_F (\dv_F \alpha)\omega_F = \int_{\partial F} \alpha(\nu_F) \omega_{\partial F} + \int_F H \alpha(n) \omega_F
\end{equation}
for any one-form $\alpha$, where $\nu_F$ is the outward unit normal to $\partial F$.

\begin{proposition} \label{prop:sffdot}
	Let $g(t)$ be a family of smooth Riemannian metrics with time derivative $\frac{\partial}{\partial t}g =: \sigma$.  Let $F$ be a time-independent hypersurface with unit normal $n$, induced volume form $\omega_F$, second fundamental form $\sff$, mean curvature $H$, and trace-reversed second fundamental form $\overline{\sff}$.  Let $\rho$ be an arbitrary time-independent symmetric $(0,2)$-tensor field.  We have
	\begin{equation} \label{codim1dot0}
		\frac{\partial}{\partial t} \left( \left\langle 2\overline{\sff},  \left.\rho\right|_F \right\rangle \omega_F \right)
		= \Big( \langle \mathbb{S}_F( (\Tr\sigma) \sff + \nabla_n\sigma - 2\nabla_F(\sigma(n,\cdot))), \rho|_F \rangle - 2 (\sigma|_F) : \overline{\sff} : (\rho|_F) \Big) \omega_F,
	\end{equation}
	where
	\begin{equation} \label{SF}
		\mathbb{S}_F\rho = \rho|_F - g|_F \Tr(\rho|_F).
	\end{equation}
\end{proposition}
\begin{proof}
	For any pair of time-independent vector fields $X$ and $Y$ that are tangent to $F$, we have
	\[
	\frac{\partial}{\partial t} \sff(X,Y) = \frac{\partial}{\partial t} g(\nabla_X n, Y) = - \frac{\partial}{\partial t} g(n,\nabla_X Y).
	\]
	Lemmas 2.4 and 2.5 from~\cite{gawlik2023scalar} tell us that
	\[
	\frac{\partial}{\partial t} g(n,\cdot) = \frac{1}{2}\sigma(n,n)g(n,\cdot)
	\]
	and
	\[
	\frac{\partial}{\partial t} \nabla_X Y = \frac{1}{2} \left( (\nabla_X \sigma)(Y,\cdot) + (\nabla_Y \sigma)(X,\cdot) - (\nabla \sigma)(X,Y) \right)^\sharp.
	\]
	Thus,
	\begin{align*}
		2 \frac{\partial}{\partial t} \sff(X,Y) &= -\sigma(n,n)g(n,\nabla_X Y) - (\nabla_X \sigma)(Y,n) - (\nabla_Y \sigma)(X,n) + (\nabla_n \sigma)(X,Y) \\
		&= \sigma(n,n) \sff(X,Y) - (\nabla_X \sigma)(Y,n) - (\nabla_Y \sigma)(X,n) + (\nabla_n \sigma)(X,Y).
	\end{align*}
	If we denote $(\sym \rho)(X,Y) := \frac{1}{2} \left( \rho(X,Y) + \rho(Y,X) \right)$ and use dots to denote time derivatives, then we can write the above equation as
	\begin{equation} \label{sffdot}
		2\dot{\sff} = \sigma(n,n)\sff + (\nabla_n \sigma)|_F - 2 \sym \left( (\nabla_F \sigma)(n,\cdot)|_F \right).
	\end{equation}
	Taking inner products with the induced metric $g|_F$, we obtain
	\begin{align*}
		2\dot{H} 
		&= 2\frac{\partial}{\partial t} \langle \sff, g|_F \rangle = 2\langle \dot{\sff}, g|_F \rangle - 2\langle \sff, \sigma|_F \rangle = \Tr\left( \sigma(n,n)\sff + (\nabla_n \sigma - 2(\nabla_F \sigma)(n,\cdot) )|_F \right) - 2\langle \sff, \sigma|_F \rangle.
	\end{align*}
	The minus sign in the second equality above follows from the fact that in coordinates, $\langle \sff, g|_F \rangle = \sff_{ij} (g|_F)^{ij}$, where $(g|_F)^{ij}$ are the components of $(g|_F)^{-1}$.  It follows that the trace-reversed second fundamental form $\overline{\sff} = \sff - Hg|_F$ satisfies
	\begin{align*}
		2\dot{\overline{\sff}} 
		&= 2\dot{\sff} - 2\dot{H} g|_F - 2H \sigma|_F = \mathbb{S}_F (\sigma(n,n)\sff + \nabla_n \sigma - 2\sym( (\nabla_F \sigma)(n,\cdot) ) ) + 2\langle \sff, \sigma|_F \rangle g|_F - 2H\sigma|_F.
	\end{align*}
	Now we are ready to compute the time derivative of $\left\langle 2\overline{\sff},  \left.\rho\right|_F \right\rangle \omega_F = 2 (g|_F)^{ij} \overline{\sff}_{jk} (g|_F)^{kl} \rho_{li}\, \omega_F$.  Using the fact that $\dot{\omega}_F = \frac{1}{2}\Tr(\sigma|_F)\, \omega_F$, we obtain
	\begin{align}
		\frac{\partial}{\partial t} \left( \left\langle 2\overline{\sff},  \left.\rho\right|_F \right\rangle\, \omega_F \right)
		&= \left\langle 2\dot{\overline{\sff}},  \left.\rho\right|_F \right\rangle \omega_F - 4 \left( (\sigma|_F) : \overline{\sff} : (\rho|_F) \right) \omega_F + \left\langle 2\overline{\sff},  \left.\rho\right|_F \right\rangle \dot{\omega}_F \nonumber \\
		&= \bigg( \big\langle \mathbb{S}_F (\sigma(n,n)\sff + \nabla_n \sigma - 2(\nabla_F \sigma)(n,\cdot) ), \rho|_F \big\rangle + 2\langle \sff, \sigma|_F \rangle \Tr(\rho|_F) - 2 H \langle \sigma|_F, \rho|_F \rangle \nonumber \\
		&\quad\quad - 4 (\sigma|_F) : \overline{\sff} : (\rho|_F) + \left\langle \overline{\sff},  \left.\rho\right|_F \right\rangle \Tr(\sigma|_F) \bigg) \omega_F. \label{codim1dot1}
	\end{align}
	This simplifies to~\eqref{codim1dot0} upon using the identities
	\[
	\left\langle \overline{\sff},  \left.\rho\right|_F \right\rangle \Tr(\sigma|_F) = \left\langle \mathbb{S}_F( \Tr(\sigma|_F) \sff), \rho|_F \right\rangle,
	\]
	\[
	\sigma(n,n) + \Tr(\sigma|_F) = \Tr\sigma,
	\]
	and
	\begin{align}
		\langle \mathbb{S}_F( \nabla_F(\sigma(n,\cdot)) - (\nabla_F\sigma)(n,\cdot) ), \rho|_F \rangle 
		&= \langle ( \nabla_F(\sigma(n,\cdot)) - (\nabla_F\sigma)(n,\cdot) )|_F, \mathbb{S}_F \rho \rangle \nonumber \\
		&= (\sigma|_F) : \sff : \mathbb{S}_F\rho \nonumber \\
		&= (\sigma|_F) : \sff : (\rho|_F) - \langle \sigma|_F, \sff \rangle \Tr(\rho|_F)  \label{gradsigman} \\
		&= (\sigma|_F) : \overline{\sff} : (\rho|_F) + H \langle \sigma|_F, \rho|_F \rangle - \langle \sigma|_F, \sff \rangle \Tr(\rho|_F). \nonumber
	\end{align}
\end{proof}

\section{Distributional densitized Einstein tensor} \label{sec:curvature}

In this section, we shift our focus away from smooth Riemannian metrics and consider instead a Regge metric $g$ on a simplicial triangulation $\mathcal{T}$ of a polyhedral domain $\Omega \subset \mathbb{R}^N$.  

Let us recall what this means.  Let $\mathcal{T}^k$ denote the set of all $k$-simplices in $\mathcal{T}$, and let $\mathring{\mathcal{T}}^k$ denote the subset of $\mathcal{T}^k$ consisting of $k$-simplices that are not contained in $\partial\Omega$.  We call such simplices \emph{interior simplices}.  We call $(N-1)$-simplices \emph{faces}.

A metric $g$ is called a Regge metric if $\left.g\right|_T$ is a smooth Riemannian metric on each $T \in \mathcal{T}^N$ and the induced metric $\left.g\right|_F$ is single-valued on each $F \in \mathring{\mathcal{T}}^{N-1}$ (and consequently the induced metric is single-valued on all lower-dimensional simplices in $\mathcal{T}$).

On each $T \in \mathcal{T}^N$, we denote by $G_T$ the Einstein tensor associated with $\left.g\right|_T$.  On an interior face $F \in \mathring{\mathcal{T}}^{N-1}$ that lies on the boundary of two $N$-simplices $T^+$ and $T^-$, the second fundamental form on $F$, as measured by $\left.g\right|_{T^+}$, generally differs from that measured by $\left.g\right|_{T^-}$.  We denote by $\llbracket \sff \rrbracket_F$ the jump in the second fundamental form across $F$.  More precisely,
\[
\llbracket \sff \rrbracket_F(X,Y) = \left.g\right|_{T^+}(\nabla_X n^+, Y) + \left.g\right|_{T^-}(\nabla_X n^-, Y)
\]
for any vectors $X,Y$ tangent to $F$, where $n^\pm$ points outward from $T^\pm$, has unit length with respect to $\left.g\right|_{T^\pm}$, and is $\left.g\right|_{T^\pm}$-orthogonal to $F$.  We adopt similar notation for the jumps in other quantities across $F$.  For instance, $\llbracket H \rrbracket_F$ denotes the jump in the mean curvature across $F$.  We sometimes drop the subscript $F$ when there is no danger of confusion.  If $F$ is contained in $\partial\Omega$, then we define the jump in a field $v$ across $F$ to be simply $\llbracket v \rrbracket_F = \left.v\right|_F$.

On each $S \in \mathcal{T}^{N-2}$, the \emph{angle defect} along $S$ is
\[
\Theta_S = m_S \pi - \displaystyle\sum_{\substack{T \in \mathcal{T}^N \\ T \supset S}} \theta_{ST}, \qquad m_S = 
\begin{cases}
	2, & \mbox{ if } S \in \mathring{\mathcal{T}}^{N-2}, \\
	1, & \mbox{ if } S \subset \partial\Omega, \\
\end{cases}
\]
where $\theta_{ST}$ denotes the dihedral angle formed by the two faces of $T$ that contain $S$, as measured by $\left.g\right|_T$. Generally this angle may vary along $S$.  If $F^+$ and $F^-$ are the two faces of $T$ that contain $S$, and if $n^\pm$ denotes the unit normal to $F^\pm$ with respect to $\left.g\right|_T$ pointing outward from $T$, then 
\[
\cos\theta_{ST} = -\left.g\right|_T(n^+,n^-).
\]

On a subdomain $D \subseteq \Omega$ or a simplex $D \in \mathcal{T}$, let $W^{s,p}(D)$ denote the Sobolev space of differentiability index $s \ge 0$ and integrability index $p \in [1,\infty]$.  Let $H^s(D)=W^{s,2}(D)$.  Let $H^2 S^0_2(D)$ denote the space of symmetric $(0,2)$-tensor fields on $D$ whose components belong to $H^2(D)$.  (Later we will also make use of the subspace $H^2_0 S^0_2(D)$ of symmetric $(0,2)$-tensor fields on $D$ whose components belong to $H^2_0(D)$, the closure of $C^\infty_0(D)$ in $H^2(D)$.)  We define
\begin{equation}
	\label{Sigmaspace}
	\Sigma:=\{ \text{symmetric }(0,2)\text{-tensor fields } \rho\,\vert\, \rho|_T\in H^2 S^0_2(T)\,\forall T \in \mathcal{T}^N, \, \rho|_F \text{ is single-valued } \forall F \in \mathring{\mathcal{T}}^{N-1} \}.
\end{equation}

Note that the second condition in \eqref{Sigmaspace} does not mean that all components of $\rho$ must be single-valued on $F$; only the tangential-tangential components of $\rho$ must be.

The Einstein curvature of $g$, defined below, will be thought of as an element of $\Sigma'$, the dual of $\Sigma$.  We denote the duality pairing between elements of $\Sigma'$ and elements of $\Sigma$ by $\llangle \cdot, \cdot \rrangle$.  We use this notation for other duality pairings as well; the spaces involved will be clear from the context.

\begin{definition} \label{def:distein}
	Let $g$ be a Regge metric.  The \emph{distributional densitized Einstein curvature} of $g$ is the linear functional $(G\omega)_{\rm dist}(g) \in \Sigma'$ defined by
	\begin{equation} \label{distcurv}
		\llangle (G\omega)_{\rm dist}(g), \rho \rrangle = \sum_{T \in \mathcal{T}^N} \int_T \langle G_T, \rho \rangle \,\omega_T + \sum_{F \in \mathcal{T}^{N-1}} \int_F \left\langle \llbracket \overline{\sff} \rrbracket_F, \rho|_F \right\rangle\, \omega_F - \sum_{S \in \mathcal{T}^{N-2}} \int_S \langle \Theta_S g|_S, \rho|_S \rangle\, \omega_S   
	\end{equation}
	for every $\rho \in \Sigma$.
\end{definition}

\begin{remark} \label{remark:ringsum}
	In the sequel, we will consistently use the letters $T$, $F$, and $S$ to refer to simplices of dimension $N$, $N-1$, and $N-2$, respectively.  We will therefore write $\sum_T$, $\sum_F$, and $\sum_S$ in place of $\sum_{T \in \mathcal{T}^N}$, $\sum_{F \in \mathcal{T}^{N-1}}$, and $\sum_{S \in \mathcal{T}^{N-2}}$, respectively.  When we wish to sum over \emph{interior} simplices of a given dimension, we put a ring on top of the summation symbol.  Thus, for example, $\mathring{\sum}_F$ is shorthand for $\sum_{F \in \mathring{\mathcal{T}}^{N-1}}$.
\end{remark}

The remainder of this section is structured as follows.  We study the behavior of  $(G\omega)_{\rm dist}(g)$ under deformations of $g$ in Section~\ref{sec:evolution_distein}, leading to a formula for its linearization in Proposition~\ref{prop:disteindot}.   Then, in Sections~\ref{sec:rel_scalar_curv} and~\ref{subsec:distr_cov_einop}, we make some observations that help to shed further light on Definition~\ref{def:distein} and Proposition~\ref{prop:disteindot}.

\subsection{Evolution of the distributional densitized Einstein tensor} \label{sec:evolution_distein}

The goal of this subsection is to understand how the distributional densitized Einstein tensor~(\ref{distcurv}) behaves under deformations of the metric.  To do this, let us consider a one-parameter family of Regge metrics $g(t)$ with time derivative
\[
\sigma = \frac{\partial}{\partial t}g.
\] 
We aim to compute
\[
\frac{d}{dt} \llangle (G\omega)_{\rm dist}(g(t)), \rho \rrangle
\]
with $\rho \in \Sigma$ arbitrary.  We will do this with the help of Proposition~\ref{prop:eindot}, Proposition~\ref{prop:sffdot}, and the following formula for the rate of change of the angle defect~\cite[Equation (18)]{gawlik2023scalar}:
\begin{equation} \label{angledefectdot}
	\dot{\Theta}_S = \frac{1}{2} \sum_{F \supset S} \llbracket \sigma(n,\nu) \rrbracket_F.
\end{equation}
In this formula, the sum is over all $(N-1)$-simplices $F$ that contain $S$, $n$ denotes the unit normal to $F$ (which differs on either side of $F$), and $\nu$ denotes the unit vector that is simultaneously tangent to $F$ and orthogonal to $S$ (which is single-valued on $F$), all with respect to the Regge metric $g$.  Our sign convention is as follows. If $F \in \mathring{\mathcal{T}}^{N-1}$, and if $T^+$ and $T^-$ are the two $N$-simplices containing $F$, then
\begin{equation} \label{sigmajump}
	\llbracket \sigma(n,\nu) \rrbracket_F = \sigma^+(n^+,\nu) + \sigma^-(n^-,\nu),
\end{equation}
where $\sigma^{\pm} = \sigma|_{T^\pm}$, $n^\pm$ points outward from $T^\pm$, and $\nu$ points into $F$ from $S$.  If $F \subset \partial\Omega$, then our convention is the same except for the fact that only one term is present on the right-hand side of~\eqref{sigmajump} because only one $N$-simplex contains $F$.

The formula for $\frac{d}{dt} \llangle (G\omega)_{\rm dist}(g(t)), \rho \rrangle$ that we will soon state can be regarded as a distributional version of Proposition~\ref{prop:eindot}.  The formula is linear in $\sigma$ (by the chain rule) and linear in $\rho$ (by definition), so it is a bilinear form in $\sigma$ and $\rho$ that depends on the Regge metric $g$.  In fact it is symmetric in $\sigma$ and $\rho$ because it is (minus) the second variation of the distributional densitized scalar curvature $\llangle (R\omega)_{\rm dist}(g), 1 \rrangle$; see Section~\ref{sec:rel_scalar_curv}. 
We will write this symmetric bilinear form as a sum of two symmetric bilinear forms, one of which corresponds to a distributional version of $\langle \ein \sigma, \rho \rangle$ and the other of which corresponds to the remaining terms in~\eqref{eindot}.  (See Lemmas~\ref{lemma:trAhBh} and~\ref{lem:eucl_distr_ein} for more insight into how the terms are partitioned.)

\begin{proposition} \label{prop:disteindot}
	Let $g(t)$ be a time-dependent Regge metric with time derivative $\sigma = \frac{\partial}{\partial t} g$.  Then for every $\rho \in \Sigma$,
	\[
	\frac{d}{dt} \left\llangle (G\omega)_{\rm dist}(g(t)), \rho \right\rrangle = B_h(g; \sigma, \rho) + A_h(g; \sigma, \rho),
	\]
	where
	\begin{align} 
		2B_h(g; \sigma, \rho) &= \sum_T \int_T \langle 2 \ein \sigma, \rho \rangle \,\omega_T \nonumber \\
		&\quad + \sum_F \int_F \left\langle \left\llbracket \mathbb{S}_F(\sigma(n,n)\sff + \nabla_n \sigma - (\nabla_F\sigma)(n,\cdot) - \nabla_F(\sigma(n,\cdot))) \right\rrbracket, \rho|_F \right\rangle \omega_F \label{Bhsimpler} \\
		&\quad - \sum_S \int_S \sum_{F \supset S} \llbracket \sigma(n,\nu) \rrbracket_F \Tr(\rho|_S) \nonumber \,\omega_S
	\end{align}	
	and
	\begin{align}
		2&A_h(g;\sigma,\rho) = \sum_T \int_T \bigg( 2\sigma : \Riem : \rho + \langle \Ric, \sigma \rangle \Tr \rho + \langle \Ric, \rho \rangle \Tr \sigma  + R \langle J\sigma, \rho \rangle - 2\sigma:\Ric:\rho \bigg) \omega_T \nonumber \\
		& + \sum_F \int_F \bigg( -3 (\sigma|_F) : \llbracket \overline{\sff} \rrbracket : (\rho|_F) +  \langle \llbracket \overline{\sff} \rrbracket, \sigma|_F \rangle \Tr(\rho|_F) + \Tr(\sigma|_F) \langle \llbracket \overline{\sff} \rrbracket, \rho|_F \rangle - \llbracket H \rrbracket \langle \mathbb{S}_F \sigma, \rho|_F \rangle \bigg) \omega_F \nonumber  \\
		& + \sum_S \int_S \bigg( 2\Theta_S \langle \sigma|_S, \rho|_S \rangle - \Theta_S \Tr(\sigma|_S) \Tr(\rho|_S) \bigg) \omega_S. \label{Ah}
	\end{align}
\end{proposition}
\begin{remark}
	In view of~\eqref{gradsigman}, another way of writing~(\ref{Bhsimpler}) is
	\begin{align}
		2&B_h(g; \sigma, \rho) = \sum_T \int_T \langle 2 \ein \sigma, \rho \rangle \,\omega_T + \sum_{F} \int_F \left\langle \left\llbracket \sigma(n,n)\overline{\sff} + \mathbb{S}_F(\nabla_n \sigma - 2(\nabla_F\sigma)(n,\cdot)) \right\rrbracket, \rho|_F \right\rangle \omega_F \nonumber \\
		&\quad + \sum_{F} \int_F \bigg( \langle \sigma|_F, \llbracket \sff \rrbracket \rangle \Tr(\rho|_F) -(\sigma|_F) : \llbracket \sff \rrbracket : (\rho|_F) \bigg) \omega_F - \sum_{S} \int_S \sum_{F \supset S} \llbracket \sigma(n,\nu) \rrbracket_F \Tr(\rho|_S)\, \omega_S. \label{Bh} 
	\end{align}
\end{remark}
\begin{remark}
	The bilinear form $A_h(g; \cdot,\cdot)$ is manifestly symmetric.  We argue in Section~\ref{sec:rel_scalar_curv} that $A_h(g;\cdot,\cdot) + B_h(g;\cdot,\cdot)$ is also symmetric, so $B_h(g;\cdot,\cdot)$ is symmetric as well. (One can also directly prove the symmetry of $B_h(g;\cdot,\cdot)$ using integration by parts.)
\end{remark}
\begin{proof}[Proof of Proposition~\ref{prop:disteindot}]
	We use Proposition~\ref{prop:eindot} to differentiate the integrals over codimension-0 simplices, Proposition~\ref{prop:sffdot} to differentiate the integrals over codimension-1 simplices, and~\eqref{angledefectdot} to differentiate the integrals over codimension-2 simplices.  Proposition~\ref{prop:eindot} immediately yields the codimension-0 terms in~\eqref{Bhsimpler} and~\eqref{Ah}.  For the codimension-2 terms, we use that fact that $\langle \Theta_S g|_S, \rho|_S \rangle = \Theta_S \Tr(\rho|_S)$ to compute
	\[
	\frac{\partial}{\partial t}( 2 \langle\Theta_S g|_S, \rho|_S \rangle\, \omega_S )
	= 2\dot{\Theta}_S \Tr(\rho|_S) \,\omega_S + 2\Theta_S \left( \frac{\partial}{\partial t} \Tr(\rho|_S) \right) \omega_S +2\Theta_S \Tr(\rho|_S) \,\dot{\omega}_S.
	\]
	Bearing in mind that the coordinate expression for $\Tr(\rho|_S)$ involves the inverse metric $(g|_S)^{-1}$, we get
	\[
	\frac{\partial}{\partial t} \Tr(\rho|_S) = -\langle \sigma|_S, \rho|_S \rangle.
	\]
	Together with $\dot{\omega}_S = \frac{1}{2} \Tr(\sigma|_S) \omega_S$ and~\eqref{angledefectdot}, we obtain
	\begin{align*}
		\frac{\partial}{\partial t} ( 2\langle\Theta_S g|_S, \rho|_S \rangle\, \omega_S )
		=\left( \sum_{F \supset S} \llbracket \sigma(n,\nu) \rrbracket_F \Tr(\rho|_S) - 2\Theta_S\langle \sigma|_S, \rho|_S \rangle + \Theta_S \Tr(\sigma|_S)\Tr(\rho|_S) \right) \omega_S,
	\end{align*}
	as desired.  
	The only thing left to check is that the codimension-1 terms in~\eqref{Bhsimpler} and~\eqref{Ah} match those in Proposition~\ref{prop:sffdot}.  Equivalently, we can check that the codimension-1 terms in~\eqref{Bh} and~\eqref{Ah} match those in~\eqref{codim1dot1}.  This requires us to check that
	\begin{align*}
		&\left\llbracket -4(\sigma|_F) : \overline{\sff} : (\rho|_F) + 2\langle \sff, \sigma|_F \rangle \Tr(\rho|_F) - 2H \langle \sigma|_F, \rho|_F \rangle \right\rrbracket \\
		&= \langle \sigma|_F, \llbracket \sff \rrbracket \rangle \Tr(\rho|_F) - (\sigma|_F) : \llbracket\sff\rrbracket : (\rho|_F) - 3(\sigma|_F) : \llbracket \overline{\sff}\rrbracket  : (\rho|_F) +  \langle \llbracket \overline{\sff} \rrbracket, \sigma|_F \rangle \Tr(\rho|_F) - \llbracket H \rrbracket \langle \mathbb{S}_F \sigma, \rho|_F \rangle.
	\end{align*}
	Rearranging, we must show that
	\begin{align*}
		\left\llbracket (\sigma|_F) : (\sff-\overline{\sff}) : (\rho|_F) + \langle \sff-\overline{\sff}, \sigma|_F \rangle \Tr(\rho|_F) \right\rrbracket = \left\llbracket 2H\langle \sigma|_F, \rho|_F \rangle - H \langle \mathbb{S}_F \sigma, \rho|_F \rangle \right\rrbracket.
	\end{align*}
	Since both sides equal $\llbracket H\langle \sigma|_F, \rho|_F \rangle + H \Tr(\sigma|_F)\Tr(\rho|_F) \rrbracket$, the proof is complete.
\end{proof}

\subsection{Relation to distributional densitized scalar curvature} \label{sec:rel_scalar_curv}

As discussed in Section~\ref{sec:intro}, Definition~\ref{def:distein} is motivated by a formula for the variation of the total scalar curvature under deformations of the metric.  We will elaborate on this motivation below, keeping in mind that in Section~\ref{sec:intro}, the discussion was restricted to compact domains without boundary.  Here we abandon the boundaryless assumption and consider, as above, a polyhedral domain $\Omega$ with boundary.

The following notion of scalar curvature for Regge metrics was put forth in~\cite{gawlik2023scalar}.  Let
\[
V=\{v \in H^1_0(\Omega) \mid \forall T \in \mathcal{T}^N, \left.v\right|_T \in H^2(T)\}
\]
and let $g$ be a Regge metric.  In~\cite{gawlik2023scalar} we defined the \emph{distributional densitized scalar curvature} of $g$ to be the linear functional $(R\omega)_{\rm dist}(g) \in V'$ given by
\[
\llangle (R\omega)_{\rm dist}(g),v\rrangle = \sum_{T} \int_T R_T v \,\omega_T + 2\mathring{\sum_{F}} \int_F \llbracket H \rrbracket v \,\omega_F + 2\mathring{\sum_{S}} \int_S \Theta_S v \,\omega_S, \quad \forall v \in V.
\]
In the present paper, it will be convenient to extend this definition so that we can choose $v=1$ as a test function.  Accordingly, we define
\[
\widetilde{V} =\{v \in H^1(\Omega) \mid \forall T \in \mathcal{T}^N, \left.v\right|_T \in H^2(T)\},
\]
and set
\begin{equation} \label{distcurv_bnd}
	\llangle (R\omega)_{\rm dist}(g),v\rrangle = \sum_{T} \int_T R_T v \,\omega_T + 2\sum_{F} \int_F \llbracket H \rrbracket v \,\omega_F + 2\sum_{S} \int_S \Theta_S v \,\omega_S, \quad \forall v \in \widetilde{V}.
\end{equation}
Recall that our convention is to interpret $\llbracket H \rrbracket_F$ as $H$ (the mean curvature measured by $g|_T$, where $T \supset F$) if $F \subset \partial\Omega$, and to interpret $\Theta_S$ as $\pi - \sum_{T\supset S} \theta_{ST}$ (as opposed to $2\pi - \sum_{T\supset S} \theta_{ST}$) if $S \subset \partial\Omega$.  One can check that this generalized definition of $(R\omega)_{\rm dist}(g)$ yields the Gauss-Bonnet formula $\frac{1}{2}\llangle (R\omega)_{\rm dist}(g),1\rrangle = 2\pi \chi(\Omega)$ in dimension $N=2$.  It is also consistent with classical definitions of curvature for piecewise constant $g$ on bounded domains of arbitrary dimension~\cite{cheeger1984curvature}.

In~\cite{gawlik2023scalar}, we computed the variation of $\llangle (R\omega)_{\rm dist}(g), v \rrangle$ under deformations of the metric, with $v \in V$ fixed.  A straightforward generalization of~\cite[Theorem 3.6]{gawlik2023scalar} to $v \in \widetilde{V}$ yields the formula
\[
\frac{d}{dt}\llangle (R\omega)_{\rm dist}(g(t)),v\rrangle = b_h(g;\sigma,v) - a_h(g;\sigma,v), \quad \forall v \in \widetilde{V},
\]
where $\sigma=\frac{\partial}{\partial t}g$,
\begin{align*}
	b_h(g;\sigma,v) &= \sum_T \int_T (\dv\dv\mathbb{S}\sigma)v\, \omega_T - \sum_{F} \int_F \left\llbracket (\dv\mathbb{S}\sigma)(n) + \dv_F\left(\sigma(n,\cdot) \right) - H \sigma(n,n) \right\rrbracket v\, \omega_F \\
	&\quad + \sum_{S} \int_S \sum_{F \supset S} \llbracket  \sigma(n,\nu) \rrbracket_F v\,\omega_S \\
	&= \sum_T \int_T \langle \mathbb{S}\sigma, \nabla\nabla v \rangle \,\omega_T - \sum_{F} \int_F \mathbb{S}\sigma(n,n) \llbracket \nabla_n v \rrbracket \,\omega_F,
\end{align*}
$\mathbb{S}\sigma = \sigma - g\Tr\sigma$, and
\begin{align*}
	a_h(g;\sigma,v) 
	&= \sum_T \int_T \langle G_T, \sigma \rangle\, v\, \omega_T + \sum_{F} \int_F \langle \llbracket \overline{\sff} \rrbracket, \sigma|_F \rangle\, v \,\omega_F - \sum_{S} \int_S \langle \Theta_S g|_S,  \sigma|_S \rangle v \,\omega_S.
\end{align*}
The fact that the two formulas for $b_h$ given above are equal follows from Lemma 3.4 in~\cite{gawlik2023scalar}.  In that lemma, the sums over $F$ and $S$ (in the first formula for $b_h$ above) are sums over interior codimension-1 and codimension-2 simplices, but here they are sums over all codimension-1 and codimension-2 simplices because $v|_{\partial\Omega}$ is not necessarily zero.

Since $b_h(g; \sigma,1) = 0$, we have
\[
\llangle (G\omega)_{\rm dist}(g), \sigma \rrangle = a_h(g;\sigma,1) = -\left.\frac{d}{dt}\right|_{t=0} \llangle (R\omega)_{\rm dist}(g+t\sigma), 1 \rrangle
\]
for any fixed Regge metric $g$.  It follows that
\[
B_h(g;\rho,\sigma) + A_h(g;\rho,\sigma) = \left.\frac{d}{ds}\right|_{s=0} \llangle (G\omega)_{\rm dist}(g+s \rho), \sigma \rrangle = -\left.\frac{d^2}{dsdt}\right|_{s=t=0} \llangle (R\omega)_{\rm dist}(g+s\rho+t\sigma), 1 \rrangle.
\]
That is, $B_h(g;\cdot,\cdot) + A_h(g;\cdot,\cdot)$ is the second variation of $-\llangle (R\omega)_{\rm dist}(g), 1 \rrangle$.  Thus it is a symmetric bilinear form.

\begin{remark}
	One should think of~\eqref{distcurv_bnd} as a distributional version of
	\[
	\int_\Omega R v \omega + 2 \int_{\partial \Omega} H v \omega_{\partial\Omega}.
	\]
	When $v=1$, this is the Einstein-Hilbert functional with the Gibbons-Hawking-York boundary term included \cite{gibbons1993action,york1972role}.  Its variation is
	\begin{align*}
		\frac{d}{dt} &\left( \int_\Omega R(g(t)) v \omega(g(t)) + 2 \int_{\partial \Omega} H(g(t)) v \omega_{\partial\Omega}(g(t)) \right) \\
		&= \int_\Omega (\dv\dv\mathbb{S}\sigma - \langle G, \sigma \rangle) v \omega - \int_{\partial\Omega}  \left( \langle \overline{\sff}, \sigma|_{\partial\Omega} \rangle + (\dv\mathbb{S}\sigma)(n) + \dv_{\partial\Omega}(\sigma(n,\cdot)) - H\sigma(n,n) \right) v \omega_{\partial\Omega},
	\end{align*}
	where $\sigma = \frac{\partial}{\partial t} g$.  This follows from~\eqref{curvdot} and~\cite[Proposition 2.2]{gawlik2023scalar}.  We can write this result as 
	\[
	\frac{d}{dt} \left( \int_\Omega R(g(t)) v \omega(g(t)) + 2 \int_{\partial \Omega} H(g(t)) v \omega_{\partial\Omega}(g(t)) \right) = b(g;\sigma,v) - a(g;\sigma,v),
	\]
	where
	\begin{align*}
		b(g;\sigma,v) &= \int_\Omega (\dv\dv \mathbb{S}\sigma) v \omega - \int_{\partial\Omega}  \left( (\dv\mathbb{S}\sigma)(n) + \dv_{\partial\Omega}(\sigma(n,\cdot)) - H\sigma(n,n) \right) v \omega_{\partial\Omega}, \\
		a(g;\sigma,v) &= \int_\Omega \langle G, \sigma \rangle v \omega + \int_{\partial\Omega} \langle \overline{\sff}, \sigma|_{\partial\Omega} \rangle v \omega_{\partial\Omega}.
	\end{align*}
	Notice that 
	\[
	b(g;\sigma,1) = 0
	\] 
	because $\int_\Omega (\dv\dv \mathbb{S}\sigma) \omega$ cancels with $- \int_{\partial\Omega} (\dv\mathbb{S}\sigma)(n) \omega_{\partial\Omega}$ and $\dv_{\partial\Omega}(\sigma(n,\cdot)) - H\sigma(n,n)$ integrates to zero by~\eqref{surfacestokes}.  Therefore $b$ and $a$ should be regarded as the smooth counterparts of $b_h$ and $a_h$.
\end{remark}

The following lemma shows that $(R\omega)_{\rm dist}$, $b_h$, and $a_h$ are, in a certain sense, ``traces'' of $(G\omega)_{\rm dist}$, $B_h$, and $A_h$.
\begin{lemma} \label{lemma:trAhBh}
	Let $g$ be a Regge metric.  For any $\sigma \in \Sigma$ and any $v \in \widetilde{V}$, we have
	\begin{align}
		\llangle (G\omega)_{\rm dist}(g), vg \rrangle &= -\left(\frac{N-2}{2}\right) \llangle (R\omega)_{\rm dist}(g), v \rrangle, \label{trG} \\
		B_h(g; \sigma, vg) &= B_h(g; vg, \sigma) = -\left( \frac{N-2}{2}\right) b_h(g; \sigma,v), \label{trBh} \\
		A_h(g; \sigma, vg) &= A_h(g; vg, \sigma) = \left( \frac{N-4}{2} \right) a_h(g; \sigma, v) \label{trAh}.
	\end{align}
\end{lemma}
\begin{proof}
	See Appendix~\ref{app:proof_lemma_trAhBh} for a proof of~(\ref{trBh}-\ref{trAh}).  A proof of~\eqref{trG} is given in~\cite[Remark 3.11]{gawlik2023scalar} for boundaryless domains, and that proof extends easily to bounded domains.
\end{proof}
\begin{remark}
	Lemma~\ref{lemma:trAhBh} is consistent with the fact that if $g$ varies with $t$ and $\sigma = \frac{\partial}{\partial t}g$, then
	\begin{align*}
		\left\llangle \frac{d}{dt} (G\omega)_{\rm dist}(g), vg \right\rrangle
		&= -\llangle (G\omega)_{\rm dist}(g), v \sigma \rrangle + \frac{d}{dt} \llangle (G\omega)_{\rm dist}(g), vg \rrangle  \\
		&= -\llangle (G\omega)_{\rm dist}(g), v \sigma \rrangle - \left( \frac{N-2}{2} \right) \frac{d}{dt} \llangle (R\omega)_{\rm dist}(g), v \rrangle \\
		&=  -a_h(g;\sigma,v) - \left( \frac{N-2}{2} \right) \left( b_h(g;\sigma,v) - a_h(g;\sigma,v) \right) \\
		&= -\left( \frac{N-2}{2} \right) b_h(g;\sigma,v) + \left( \frac{N-4}{2} \right) a_h(g; \sigma, v)
	\end{align*}
	for any $v \in \widetilde{V}$.
\end{remark}

\subsection{Distributional linearized Einstein operator}
\label{subsec:distr_cov_einop}

When $g$ is the Euclidean metric $\delta$, the bilinear form $B_h(g;\cdot,\cdot)$ simplifies to
\begin{align}
	\label{eq:distr_eucl_ein}
	\begin{split}
		B_h(\delta; \sigma,\rho) = \sum_T \int_T  \ein \sigma :\rho &+ \frac{1}{2}{\sum_{F}} \int_F \left\llbracket \mathbb{S}_F(\nabla_n \sigma - 2(\nabla_F\sigma)(n,\cdot)) \right\rrbracket: \rho|_F \\
		&\quad - \frac{1}{2}{\sum_{S}} \int_S \sum_{F \supset S} \llbracket \sigma(n,\nu) \rrbracket_F \Tr(\rho|_S),
	\end{split}
\end{align}
where all differential operators and normal vectors above are with respect to the Euclidean metric, $A:B=\sum_{i,j=1}^NA_{ij}B_{ij}$ denotes the Frobenius inner product, and we have omitted the (Euclidean) volume forms for brevity.  The right-hand side of~\eqref{eq:distr_eucl_ein} is precisely the expression one encounters when computing the action of the differential operator $\ein$, interpreted in a distributional sense, on piecewise smooth symmetric tensor fields possessing tangential-tangential continuity.
\begin{lemma}
	\label{lem:eucl_distr_ein}
	For any piecewise smooth $\sigma \in \Sigma$ and any smooth symmetric $(0,2)$-tensor field $\rho$ with compact support, we have
	\[
	B_h(\delta; \sigma,\rho) = \llangle \ein_{\rm dist} \sigma, \rho \rrangle :=  \int_\Omega \sigma : \ein \rho,
	\]
	where $\ein$ is taken with respect to the Euclidean metric $\delta$.
\end{lemma}
\begin{proof}
	See Appendix~\ref{app:distr_Eucl_ein}.
\end{proof}

\section{Convergence} \label{sec:convergence}

We will now consider a family of triangulations $\{\mathcal{T}_h\}_{h>0}$ of $\Omega$ parametrized by $h = \max_{T \in \mathcal{T}_h^N} h_T$, where $h_T = \operatorname{diam}(T)$.  We address the following question:  If $g_h$ is a Regge metric on $\mathcal{T}_h$ that approximates $g$, how close is its distributional densitized Einstein tensor to the densitized Einstein tensor of $g$?  We will measure the error with the negative-order Sobolev norm
\begin{equation} \label{2norm}
	\|\sigma\|_{H^{-2}(\Omega)} = \sup_{\substack{\rho \in H^2_0 S^0_2(\Omega), \\ \rho \neq 0}} \frac{ \llangle \sigma, \rho \rrangle_{H^{-2}(\Omega),  H^2_0(\Omega)} }{ \|\rho\|_{H^2(\Omega)} }
\end{equation}
and show that it converges to zero as $h \to 0$ under suitable assumptions on $g$, $\{g_h\}_{h>0}$, and $\{\mathcal{T}_h\}_{h>0}$. The assumption we make on $\{\mathcal{T}_h\}_{h>0}$ is shape-regularity; that is, there exists a constant $C_0$ independent of $h$ such that
\[
\max_{T \in \mathcal{T}_h^N} \frac{h_T}{\varrho_T} \le C_0
\]
for all $h>0$, where $\varrho_T$ denotes the inradius of $T$. 

Throughout what follows, $\|\cdot\|_{W^{s,p}(D)}$ and $|\cdot|_{W^{s,p}(D)}$ denote the $W^{s,p}(D)$-norm and $W^{s,p}(D)$-seminorm, respectively, on a subdomain $D \subseteq \Omega$ or a simplex $D \in \mathcal{T}_h$  for $s \ge 0$ and $p \in [1,\infty]$.  We write $\|\cdot\|_{L^p(D)}=\|\cdot\|_{W^{0,p}(D)}$, $\|\cdot\|_{H^s(D)}=\|\cdot\|_{W^{s,2}(D)}$, and $|\cdot|_{H^s(D)}=|\cdot|_{W^{s,2}(D)}$.  We use the same notation when taking Sobolev norms of tensor fields, like in~\eqref{2norm}.  We understand these norms and seminorms to be with respect to the Euclidean metric $\delta$.  Occasionally we will also make use of the metric-dependent norm
\[
\|\rho\|_{L^p(D,g)} = 
\begin{cases}
	\left( \int_D |\rho|_g^p \, \omega_D(g) \right)^{1/p}, &\mbox{ if } 1 \le p < \infty, \\
	\sup_D |\rho|_g, &\mbox{ if } p=\infty,
\end{cases}
\]
where $\omega_D(g)$ is the induced volume form on $D$ and $|\rho|_g = \langle \rho, \rho \rangle_g^{1/2}$.  Note that $\|\cdot\|_{L^p(D,\delta)} = \|\cdot\|_{L^p(D)}$.

Our main result is stated in the following theorem and will be proved in the consecutive sections.  Below, we use $(G\omega)(g)$ to denote the linear functional $\rho \mapsto \int_\Omega \langle G(g), \rho \rangle_g \omega(g)$.
\begin{theorem} \label{thm:conv}
	Let $\Omega \subset \mathbb{R}^N$ be a polyhedral domain equipped with a smooth Riemannian metric $g$.  Let $\{g_h\}_{h>0}$ be a family of Regge metrics defined on a shape-regular family $\{\mathcal{T}_h\}_{h>0}$ of triangulations of $\Omega$.  Assume that $\lim_{h \to 0} \|g_h-g\|_{L^\infty(\Omega)} = 0$ and $C_1:=\sup_{h>0} \max_{T \in \mathcal{T}_h^N} \|g_h\|_{W^{2,\infty}(T)} < \infty$. Then there exist positive constants $C$ and $h_0$ such that
	\begin{equation} \label{conv}
		\begin{split}
			\| (G\omega)_{\rm dist}(g_h) - (G\omega)(g)&\|_{H^{-2}(\Omega)} \le C \left( 1 + \max_T h_T^{-2} \|g_h-g\|_{L^\infty(T)} + \max_T h_T^{-1} |g_h-g|_{W^{1,\infty}(T)}  \right) \\
			&\quad\times \left( \|g_h-g\|_{L^2(\Omega)}^2 + \sum_T h_T^2 |g_h-g|_{H^1(T)}^2 + \sum_T h_T^4 |g_h-g|_{H^2(T)}^2 \right)^{1/2}
		\end{split}
	\end{equation}
	for all $h \le h_0$.  The constants $C$ and $h_0$ depend on $N$, $\|g\|_{W^{2,\infty}(\Omega)}$, $\|g^{-1}\|_{L^\infty(\Omega)}$, $C_0$, and $C_1$.
\end{theorem}

Notice that if $g_h$ satisfies error estimates of the form 
\begin{equation} \label{interperr}
	|g_h-g|_{W^{s,p}(T)} = O(h_T^{r+1-s}), \quad (s,p) \in \{(0,\infty),(1,\infty),(0,2),(1,2),(2,2)\}
\end{equation}
for some integer $r \ge 1$, then Theorem~\ref{thm:conv} leads to an error estimate of the form $\| (G\omega)_{\rm dist}(g_h) - (G\omega)(g) \|_{H^{-2}(\Omega)} = O(h^{r+1})$.  This is typically the case when $g_h$ is a piecewise polynomial interpolant of $g$ of degree $r \ge 1$.

To make this more precise, and to accommodate slightly more natural hypotheses than~\eqref{interperr}, we follow~\cite{gawlik2023scalar} and define the notion of an optimal-order interpolation operator onto the Regge finite element space.  Recall that the Regge finite element space of degree $r \ge 0$ consists of symmetric $(0,2)$-tensor fields on $\Omega$ that are piecewise polynomial of degree at most $r$ and possess single-valued tangential-tangential components on codimension-1 simplices. 

\begin{definition}[Definition 4.2 in~\cite{gawlik2023scalar}] \label{def:optimalorder}
	Let $\mathcal{I}_h$ be a map that sends smooth symmetric $(0,2)$-tensor fields on $\Omega$ to the Regge finite element space of degree $r \ge 0$.   We say that $\mathcal{I}_h$ is an \emph{optimal-order interpolation operator} of degree $r$ if there exists a number $m \in \{0,1,\dots,N\}$ and a constant $C_2=C_2(N,r,h_T/\varrho_T,t,s)$ such that for every $p \in [1,\infty]$, every $s \in (m/p,r+1]$, every $t \in [0,s]$, and every symmetric $(0,2)$-tensor field $g$ possessing $W^{s,p}(\Omega)$-regularity, $\mathcal{I}_h g$ exists (upon continuously extending $\mathcal{I}_h$) and satisfies
	\begin{equation} \label{optimalorder}
		|\mathcal{I}_h g - g|_{W^{t,p}(T)} \le C_2 h_T^{s-t} |g|_{W^{s,p}(T)}
	\end{equation}
	for every $T \in \mathcal{T}_h^N$.  We call the number $m$ the \emph{codimension index} of $\mathcal{I}_h$.  A Regge metric $g_h$ is called an \emph{optimal-order interpolant} of $g$ having degree $r$ and codimension index $m$ if it is the image of a Riemannian metric $g$ under an  optimal-order interpolation operator having degree $r$ and codimension index $m$.  
\end{definition}

As discussed in~\cite{gawlik2023scalar}, an example of an optimal-order interpolation operator is the canonical interpolation operator onto the degree-$r$ Regge finite element space introduced in~\cite[Chapter 2]{li2018regge}.   It has codimension index $m=N-1$ because its degrees of freedom involve integrals over simplices of codimension at most $N-1$, which are generally ill-defined unless $g$ possesses $W^{s,p}(\Omega)$-regularity with $s>(N-1)/p$.

\begin{corollary} \label{cor:conv}
	Let $\Omega$, $g$, and $\{\mathcal{T}_h\}_{h>0}$ be as in Theorem~\ref{thm:conv}.  Let $\{g_h\}_{h>0}$ be a family of optimal-order interpolants of $g$ having degree $r \ge 1$ and codimension index $m$. Then there exist positive constants $C$ and $h_0$ such that
	\[
	\| (G\omega)_{\rm dist}(g_h) - (G\omega)(g)\|_{H^{-2}(\Omega)} \le C \left( \sum_T h_T^{p(r+1)} |g|_{W^{r+1,p}(T)}^p \right)^{1/p}
	\]
	for all $h \le h_0$ and all $p \in [2,\infty]$ satisfying $p > \frac{m}{r+1}$.  (We interpret the right-hand side as $C\max_T h_T^{r+1} |g|_{W^{r+1,\infty}(T)}$ if $p=\infty$.)  The constants $C$ and $h_0$ depend on the same quantities listed in Theorem~\ref{thm:conv}, as well as on $\Omega$ and $r$.
\end{corollary}

\begin{remark} \label{remark:optimalorder_general}
	The corollary above continues to hold if we relax~(\ref{optimalorder}) to the condition that
	\begin{equation} \label{optimalorder_general}
		|\mathcal{I}_h g - g|_{W^{t,p}(T)} \le C_2 h_T^{s-t} \sum_{\substack{T' : T' \cap T \neq \emptyset}} |g|_{W^{s,p}(T')},
	\end{equation}
	where the sum is over all $T' \in \mathcal{T}_h^N$ that share a subsimplex with $T$.  We will exploit this observation in Section~\ref{sec:numerical} by using an interpolant that satisfies~\eqref{optimalorder_general} but not~\eqref{optimalorder} to do our numerical experiments.
\end{remark}

In what follows, we reuse the letter $C$ to denote a positive constant that may change at each occurrence and may depend on $N$, $\|g\|_{W^{2,\infty}(\Omega)}$, $\|g^{-1}\|_{L^\infty(\Omega)}$, $C_0$, and $C_1$.

We will prove Theorem~\ref{thm:conv} using a strategy that parallels the one used in~\cite{gawlik2023scalar}.  Consider the evolving metric
\[
\gt(t) = (1-t)g + tg_h
\]
with time derivative
\[
\sigma = \frac{\partial}{\partial t}\gt(t) = g_h-g.
\]
Since $\gt(0)=g$, $\gt(1)=g_h$, and $\gt(t)$ is a Regge metric for all $t \in [0,1]$, Proposition~\ref{prop:disteindot} implies that
\begin{equation} \label{integral_rep}
	\llangle (G\omega)_{\rm dist}(g_h) - (G\omega)(g), \rho \rrangle =
	\int_0^1 B_h(\gt(t);\sigma,\rho) + A_h(\gt(t);\sigma,\rho) \, dt, \quad \forall \rho \in \Sigma.
\end{equation}
Thus, we can estimate $(G\omega)_{\rm dist}(g_h) - (G\omega)(g)$ by estimating the bilinear forms $B_h(\gt(t);\cdot,\cdot)$ and $A_h(\gt(t);\cdot,\cdot)$.

First we need to recall a few basic estimates that were discussed in~\cite[Section 4]{gawlik2023scalar}, all of which were proved or follow readily from estimates proved in~\cite[Section 4.2]{gawlik2020high}.  Assume that $\lim_{h \to 0} \|g_h-g\|_{L^\infty(\Omega)} = 0$ and $\sup_{h>0} \max_{T \in \mathcal{T}_h^N} \|g_h\|_{W^{2,\infty}(T)} < \infty$.  Then for every $h$ sufficiently small and every $t \in [0,1]$,
\begin{equation} \label{gtbound}
	\|\gt\|_{L^\infty(\Omega)} + \|\gt^{-1}\|_{L^\infty(\Omega)} + \max_{T \in \mathcal{T}_h^N} |\gt|_{W^{1,\infty}(T)}  + \max_{T \in \mathcal{T}_h^N} |\gt|_{W^{2,\infty}(T)} \le C.
\end{equation}
Furthermore,
\begin{equation} \label{Lpequiv1}
	C^{-1} \|\rho\|_{L^p(D,\gt(t_2))} \le \|\rho\|_{L^p(D,\gt(t_1))} \le C \|\rho\|_{L^p(D,\gt(t_2))}
\end{equation}
and
\begin{equation} \label{Lpequiv2}
	C^{-1} \|\rho\|_{L^p(D)} \le \|\rho\|_{L^p(D,\gt(t_1))} \le C \|\rho\|_{L^p(D)}
\end{equation}
for every $t_1,t_2 \in [0,1]$, every simplex $D \in \mathcal{T}_h$, every $p \in [1,\infty]$, every tensor field $\rho$ having finite $L^p(D)$-norm, and every $h$ sufficiently small.  We select $h_0>0$ so that~(\ref{gtbound}-\ref{Lpequiv2}) hold for all $h \le h_0$, and we tacitly use these inequalities throughout our analysis.

We will need the following additional estimates in our analysis.  Note that in what follows, we make explicit the dependencies of various quantities on the metric by either appending a subscript or by referencing the metric in parentheses.  In particular, on the boundary of any $N$-simplex $T$, we let $n_{\gt}$ denote the outward unit normal vector with respect to $\left.\gt\right|_T$.  We remark that the Euclidean length of $n_{\gt}$ is everywhere bounded above by a constant independent of $h$ and $t$, owing to~\eqref{gtbound}.
\begin{lemma}
	\label{lem:est_facet_terms}
	Let $F\in \mathring{\mathcal{T}}^{N-1}_h$, and let $T_1$, $T_2\in \mathcal{T}^N_h$ be such that $F=T_1\cap T_2$. Let $\rho \in H^2_0 S^0_2(\Omega)$.  There holds
	\begin{align}
		\|\llbracket \sff(\gt)\rrbracket\|_{L^{\infty}(F,\gt)}&\le C \|\llbracket g_h-g\rrbracket\|_{W^{1,\infty}(F)}\le C\big(\|g_h-g\|_{W^{1,\infty}(T_1)}+\|g_h-g\|_{W^{1,\infty}(T_2)}\big),\\
		\|\llbracket (\nabla_{\gt,n_{\gt}}\rho)|_F\rrbracket\|_{L^2(F,\gt)}&\le C\big( \|\llbracket g_h-g\rrbracket\|_{L^{\infty}(F)}\|\nabla_{\delta}\rho\|_{L^2(F)}+\|\llbracket g_h-g\rrbracket\|_{W^{1,\infty}(F)}\|\rho\|_{L^2(F)}\big)\nonumber\\
		&\le C\big( \|\nabla_{\delta} \rho\|_{L^2(F)}\sum_{i=1}^2\|g_h-g\|_{L^{\infty}(T_i)}+ \|\rho\|_{L^2(F)}\sum_{i=1}^2\|g_h-g\|_{W^{1,\infty}(T_i)}\big).
	\end{align}
	If $g_h$ is piecewise constant, then $\|g_h-g\|_{W^{1,\infty}}$ can be replaced by $\|g_h-g\|_{L^{\infty}}$ in both estimates.
\end{lemma}
\begin{proof}
	The first statement was proved in the proof of \cite[Lemma 4.9]{gawlik2023scalar}. For the second we proceed similarly by using the fact that $\llbracket ab\rrbracket = \llbracket a\rrbracket\{b\}+\{a\}\llbracket b \rrbracket$, where $\{\cdot\}$ denotes the average across $F$.  In Euclidean coordinates, the components of $\nabla_{\gt,n_{\gt}}\rho$ satisfy
	\begin{align*}
		\|\llbracket (\nabla_{\gt,n_{\gt}}\rho)_{ij} \rrbracket&\|_{L^2(F,\gt)} 
		\le C \|\llbracket (\nabla_{\gt,n_{\gt}}\rho)_{ij} \rrbracket\|_{L^2(F)} \\
		&= C\|\llbracket (\partial_l\rho_{ij}-\tilde{\Gamma}_{li}^k\rho_{kj}-\tilde{\Gamma}_{lj}^k\rho_{ik})n_{\gt}^l\rrbracket\|_{L^2(F)}\\
		&\le C ( \|\partial _l\rho_{ij}\llbracket n_{\gt}^l\rrbracket\|_{L^2(F)} +  \|\rho\|_{L^2(F)}(\|\llbracket \tilde{\Gamma}\rrbracket\|_{L^{\infty}(F)}\|\{ n_{\gt}\}\|_{L^{\infty}(F)} + \|\{ \tilde{\Gamma} \}\|_{L^{\infty}(F)}\|\llbracket n_{\gt}\rrbracket\|_{L^{\infty}(F)})\\
		&\le C\big( \|\nabla_{\delta}\rho\|_{L^2(F)}\|\llbracket n_{\gt}\rrbracket\|_{L^{\infty}(F)}+ \|\rho\|_{L^2(F)}(\|\llbracket g_h-g\rrbracket\|_{W^{1,\infty}(F)} + \|\llbracket n_{\gt}\rrbracket\|_{L^{\infty}(F)})\big)\\
		&\le C\big(\|\nabla_{\delta} \rho\|_{L^2(F)}\|\llbracket g_h-g\rrbracket\|_{L^{\infty}(F)} + \|\rho\|_{L^2(F)}\|\llbracket g_h-g \rrbracket\|_{W^{1,\infty}(F)}\big)\\
		&\le C\big( \|\nabla_{\delta} \rho\|_{L^2(F)}\sum_{k=1}^2\|g_h-g\|_{L^{\infty}(T_k)}+ \|\rho\|_{L^2(F)}\sum_{k=1}^2\|g_h-g\|_{W^{1,\infty}(T_k)}\big).
	\end{align*}
	Here, we used the notation $\tilde{\Gamma}^k_{ij}$ for the Christoffel symbols of the second kind associated with $\gt$, and we used the following two inequalities that follow from~\cite[Equation (43)]{gawlik2023scalar} and~\cite[Lemma 4.6]{gawlik2023scalar}, respectively:
	\begin{align}
		\|\llbracket \tilde{\Gamma} \rrbracket\|_{L^{\infty}(F)} &\leq C \|\llbracket g_h-g\rrbracket\|_{W^{1,\infty}(F)}, \\
		\|\llbracket n_{\gt}\rrbracket\|_{L^{\infty}(F)} &\le C\|\llbracket \gt \rrbracket\|_{L^\infty(F)} =  C\|\llbracket \gt-g \rrbracket\|_{L^\infty(F)} \le C\|\llbracket g_h-g \rrbracket\|_{L^\infty(F)}. \label{njump}
	\end{align} 
	
	If $g_h$ is piecewise constant there holds $\|\llbracket g_h-g\rrbracket\|_{W^{1,\infty}(F)}=\|\llbracket g_h-g\rrbracket\|_{L^{\infty}(F)}$.
\end{proof}
\noindent
We define the following mesh-dependent norms:
\begin{align}
	\nrm{\sigma}^2_2:=\sum_{T}\big(\|\sigma\|^2_{L^2(T)}+h_T^2 |\sigma|_{H^1(T)}^2\big),\qquad \nrmt{\sigma}^2_2:=\sum_{T}\big(\|\sigma\|^2_{L^2(T)}+h_T^2 |\sigma|_{H^1(T)}^2+h_T^4|\sigma|_{H^2(T)}^2\big).
\end{align}
We write, e.g., $\nrm{\sigma}_{2,T_i}$ if only element $T_i\in\mathcal{T}_h^N$ is considered in the sum.

\subsection{Convergence of the bilinear form $B_h$}
We investigate the convergence of the bilinear form $B_h(\gt;\cdot,\cdot)$ defined in~\eqref{Bh}.  Let  $\rho$ be an arbitrary member of $H^2_0 S^0_2(\Omega)$, and let $\sigma = g_h-g$.  We can use the symmetry of $B_h(\gt;\cdot,\cdot)$ to write
\begin{align}
	&B_h(\gt; \sigma, \rho)  = \sum_T \int_T \langle  \ein_{\gt} \rho, \sigma \rangle_{\gt} \omega_T(\gt) \nonumber \\ 
	&\quad + \mathring{\sum_F} \int_F \frac{1}{2}\langle \llbracket \rho(n_{\gt},n_{\gt})\overline{\sff}(\gt) + \mathbb{S}_{F,\gt}(\nabla_{\gt,n_{\gt}} \rho - 2(\nabla_{F,\gt}\rho)(n_{\gt},\cdot)) \rrbracket, \sigma|_F \rangle_{\gt} \omega_F(\gt) \nonumber \\
	&\quad + \mathring{\sum_F} \int_F \frac{1}{2}\bigg( \langle \rho|_F, \llbracket \sff(\gt) \rrbracket \rangle_{\gt} \Tr_{\gt}(\sigma|_F) -(\rho|_F) : \llbracket \sff(\gt) \rrbracket : (\sigma|_F) \bigg) \omega_F(\gt) \nonumber \\
	&\quad - \mathring{\sum_S} \int_S \frac{1}{2} \sum_{F \supset S} \llbracket \rho(n_{\gt},\nu_{\gt}) \rrbracket_F \Tr_{\gt}(\sigma|_S) \omega_S(\gt). \label{Bh2} 
\end{align}
Notice that the sums over codimension-1 simplices $F$ and codimension-2 simplices $S$ appearing above are sums over interior simplices (recall Remark~\ref{remark:ringsum}), owing to the fact that $\rho$ and its first derivatives vanish on $\partial\Omega$.

In the following we will estimate the codimension-0, codimension-1, and codimension-2 terms separately.
\begin{lemma}
	\label{lemma:Bh_vol_bound}
	There holds
	\begin{align*}
		\left|\sum_T \int_T \langle  \ein_{\gt} \rho, \sigma \rangle_{\gt} \omega_T(\gt)\right| \le C \|g_h-g\|_{L^2(\Omega)}\|\rho\|_{H^2(\Omega)}.
	\end{align*}
\end{lemma}
\begin{proof}
	This follows immediately from~\eqref{gtbound}.  It implies that
	\[
	\|\ein_{\gt} \rho\|_{L^2(T)}\le C \|\rho\|_{H^2(T)}
	\]
	and thus
	\begin{align*}
		\left| \int_T \langle  \ein_{\gt} \rho, \sigma \rangle_{\gt} \omega_T(\gt)\right|\leq \|\ein_{\gt} \rho\|_{L^2(T,\gt)}\|\sigma\|_{L^2(T,\gt)}\leq C\|\rho\|_{H^2(T)}\|\sigma\|_{L^2(T)}=C\|\rho\|_{H^2(T)}\|g_h-g\|_{L^2(T)}.
	\end{align*}
	Summing over all $T\in \mathcal{T}_h^N$ finishes the proof.
\end{proof}

\begin{lemma}
	\label{lemma:Bh_bnd_bound}
	There holds
	\begin{align}
		&\left|\mathring{\sum_F} \int_F \langle\llbracket\mathbb{S}_{F,\gt}(\nabla_{\gt,n_{\gt}} \rho - 2(\nabla_{F,\gt}\rho)(n_{\gt},\cdot)) \rrbracket, \sigma|_F \rangle_{\gt} \omega_F(\gt)\right| \le \nonumber\\
		&\qquad C\Big(\max_T \left( h_T^{-1} \|g_h-g\|_{W^{1,\infty}(T)}\right)\nrm{\rho}_2 + \max_T \left( h_T^{-1} \|g_h-g\|_{L^\infty(T)}\right)\nrm{\nabla_{\delta}\rho}_2\Big) \nrm{g_h-g}_2  		\label{eq:bh_bnd_term1}
	\end{align}
	and
	\begin{align}
		&\left|\mathring{\sum_F} \int_F \Big(\langle\llbracket\rho(n_{\gt},n_{\gt})\overline{\sff}(\gt) \rrbracket, \sigma|_F \rangle_{\gt}+\langle \rho|_F, \llbracket \sff(\gt) \rrbracket \rangle_{\gt} \Tr_{\gt}(\sigma|_F) -(\rho|_F) : \llbracket \sff(\gt) \rrbracket : (\sigma|_F)\Big) \omega_F(\gt)\right|\nonumber \\
		&\qquad\le C\max_T \left( h_T^{-1} \|g_h-g\|_{W^{1,\infty}(T)}\right) \nrm{g_h-g}_2 \nrm{\rho}_2.\label{eq:bh_bnd_term2}
	\end{align}
	If $g_h$ is piecewise constant, then $\|g_h-g\|_{W^{1,\infty}(T)}$ can be replaced by $\|g_h-g\|_{L^{\infty}(T)}$ in \eqref{eq:bh_bnd_term1}-\eqref{eq:bh_bnd_term2}.
\end{lemma}
\begin{proof}
	Consider a face $F$ shared by two $N$-simplices $T_1$ and $T_2$.  Notice that 
	\[
	\langle \llbracket \mathbb{S}_{F,\gt}\nabla_{\gt,n_{\gt}} \rho \rrbracket, \sigma|_F \rangle_{\gt} = \langle \mathbb{S}_{F,\gt} \llbracket \nabla_{\gt,n_{\gt}} \rho \rrbracket, \sigma|_F \rangle_{\gt} = \langle \llbracket (\nabla_{\gt,n_{\gt}} \rho)|_F \rrbracket, \mathbb{S}_{F,\gt} \sigma \rangle_{\gt} .
	\]
	Thus, for the first term in \eqref{eq:bh_bnd_term1}, we can use the bound $\|\mathbb{S}_{F,\gt} \sigma\|_{L^2(F,\gt)} \le C \|\sigma|_F\|_{L^2(F,\gt)}$ together with the trace inequality
	\begin{align}
		\label{eq:trace_inequ}
		\|\rho\|^2_{L^2(F)}\leq C(h_{T_1}^{-1}\|\rho\|^2_{L^2(T_1)}+h_{T_1}|\rho|^2_{H^1(T_1)})
	\end{align}
	and Lemma~\ref{lem:est_facet_terms} to write
	\begin{align*}
		&\left|\int_F \langle \llbracket \mathbb{S}_{F,\gt}\nabla_{\gt,n_{\gt}} \rho \rrbracket, \sigma|_F \rangle_{\gt}\, \omega_F(\gt)\right| \\
		&\qquad\le  C\|\llbracket(\nabla_{\gt,n_{\gt}} \rho)|_F \rrbracket\|_{L^2(F,\gt)}\|\sigma|_F\|_{L^2(F,\gt)}\\ 
		&\qquad\le C\Big(\sum_{i=1}^2\|g_h-g\|_{W^{1,\infty}(T_i)}\|\rho\|_{L^2(F)}+\sum_{i=1}^2\|g_h-g\|_{L^{\infty}(T_i)}\|\nabla_{\delta}\rho\|_{L^2(F)}\Big)\|\sigma|_F\|_{L^2(F)}\\
		&\qquad\le Ch_{T_1}^{-1}\Big(\sum_{i=1}^2\|g_h-g\|_{W^{1,\infty}(T_i)}\nrm{\rho}_{2,T_1}+\sum_{i=1}^2\|g_h-g\|_{L^{\infty}(T_i)}\nrm{\nabla_{\delta}\rho}_{2,T_1}\Big)\nrm{\sigma}_{2,T_1}.
	\end{align*}
	By the shape-regularity of $\mathcal{T}_h$, we have $C^{-1} \le h_{T_1}/h_{T_2} \le C$ for some constant $C$ independent of $h$ and $F$, so
	\begin{align*}
		\Big|\mathring{\sum_{F}}\int_F &\langle\llbracket \mathbb{S}_{F,\gt}\nabla_{\gt,n_{\gt}} \rho \rrbracket, \sigma|_F  \rangle_{\gt} \omega_F(\gt)\Big| \\ &\le C\big(\max_T \left( h_T^{-1} \|g_h-g\|_{W^{1,\infty}(T)}\right)\nrm{\rho}_2 + \max_T  \left( h_T^{-1} \|g_h-g\|_{L^{\infty}(T)}\right)\nrm{\nabla_{\delta}\rho}_2 \big)
		\nrm{g_h-g}_2 .
	\end{align*}
	The other terms in \eqref{eq:bh_bnd_term1} follow analogously. For \eqref{eq:bh_bnd_term2} we observe that
	\[
	\llbracket\rho(n_{\gt},n_{\gt})\overline{\sff}(\gt) \rrbracket = \llbracket ( \rho, n_{\gt} \otimes n_{\gt} ) \overline{\sff}(\gt) \rrbracket = ( \rho, \llbracket n_{\gt} \otimes n_{\gt} \rrbracket ) \{ \overline{\sff}(\gt) \} + ( \rho, \{ n_{\gt} \otimes n_{\gt} \} ) \llbracket \overline{\sff}(\gt) \rrbracket,
	\] 
	where $(\cdot,\cdot)$ denotes the natural pairing between $(0,2)$-tensors and $(2,0)$-tensors.  Since $\llbracket n_{\gt} \otimes n_{\gt} \rrbracket = \llbracket n_{\gt} \rrbracket \otimes \{ n_{\gt} \} + \{ n_{\gt} \} \otimes \llbracket n_{\gt} \rrbracket$, we can use Lemma~\ref{lem:est_facet_terms},~\eqref{njump}, and the trace inequality to estimate
	\begin{align*}
		&\left| \int_F \Big(\langle\llbracket\rho(n_{\gt},n_{\gt})\overline{\sff}(\gt) \rrbracket, \sigma|_F \rangle_{\gt}+\langle \rho|_F, \llbracket \sff(\gt) \rrbracket \rangle_{\gt} \Tr_{\gt}(\sigma|_F) -(\rho|_F) : \llbracket \sff(\gt) \rrbracket : (\sigma|_F)\Big) \omega_F(\gt)\right|\\
		&\qquad\leq C\big( \|\llbracket n_{\gt} \rrbracket\|_{L^\infty(F,\gt)} +  \|\llbracket \overline{\sff}(\gt)\rrbracket\|_{L^{\infty}(F,\gt)}+\|\llbracket\sff(\gt)\rrbracket\|_{L^{\infty}(F,\gt)}\big)\|\rho\|_{L^2(F,\gt)}\|\sigma|_F\|_{L^2(F,\gt)}\\
		&\qquad\leq C h_{T_1}^{-1}\sum_{i=1}^2\|g_h-g\|_{W^{1,\infty}(T_i)}\nrm{\sigma}_{2,T_1} \nrm{\rho}_{2,T_1}.
	\end{align*}
	With the same shape-regularity  argument as before, we obtain
	\begin{align*}
		&\left| \mathring{\sum_{F}}\int_F \Big(\langle\llbracket\rho(n_{\gt},n_{\gt})\overline{\sff}(\gt) \rrbracket, \sigma|_F \rangle_{\gt}+\langle \rho|_F, \llbracket \sff(\gt) \rrbracket \rangle_{\gt} \Tr_{\gt}(\sigma|_F) -(\rho|_F) : \llbracket \sff(\gt) \rrbracket : (\sigma|_F)\Big) \omega_F(\gt)\right|\\
		&\qquad\le C\max_T \left( h_T^{-1} \|g_h-g\|_{W^{1,\infty}(T)}\right) \nrm{g_h-g}_2 \nrm{\rho}_2.
	\end{align*}
\end{proof}

\begin{lemma}
	\label{lemma:Bh_bbnd_bound}
	There holds
	\begin{align*}
		&\left|\mathring{\sum_S} \int_S \sum_{F \supset S} \llbracket \rho(n_{\gt},\nu_{\gt}) \rrbracket_F \Tr_{\gt}(\sigma|_S) \omega_S(\gt)\right|\le C\max_T\left(h^{-2}_T\|g_h-g\|_{L^{\infty}(T)}\right)\,\nrmt{\rho}_2\nrmt{g_h-g}_2.
	\end{align*}
\end{lemma}
\begin{proof}
	By the shape regularity of $\mathcal{T}_h$, the number of faces attached to $S$ is bounded by a constant $C$ independent of $h$. 
	Using~\eqref{njump} and the analogous estimate $\|\llbracket \nu_{\gt} \rrbracket_F \|_{L^\infty(S)} \le C \|\llbracket g_h-g \rrbracket_F \|_{L^\infty(S)}$, we see that
	\begin{align*}
		\Big\| \sum_{F \supset S} \llbracket \rho(n_{\gt},\nu_{\gt}) \rrbracket_F\Big\|_{L^2(S)} 
		&= \Big\| \sum_{F \supset S} \rho(\llbracket n_{\gt} \rrbracket_F, \{\nu_{\gt}\}_F) + \rho(\{ n_{\gt} \}_F, \llbracket\nu_{\gt}\rrbracket_F) \Big\|_{L^2(S)} \\
		&\le C \sum_{F \supset S} \|\llbracket g_h-g\rrbracket_F\|_{L^{\infty}(S)}\|\rho\|_{L^2(S)}.
	\end{align*}
	Thus,
	\begin{align*}
		\left|\int_S \sum_{F \supset S} \llbracket \rho(n_{\gt},\nu_{\gt}) \rrbracket_F \Tr_{\gt}(\sigma|_S) \omega_S(\gt)\right|\leq C \sum_{F \supset S} \|\llbracket g_h-g\rrbracket_F\|_{L^{\infty}(S)}\|\rho\|_{L^2(S)}\|\sigma|_S\|_{L^2(S)}.
	\end{align*}
	With the codimension-2 trace inequality
	\[
	\|\rho\|_{L^2(S)}^2 \le C\left( h_{T}^{-2} \|\rho\|_{L^2(T)}^2 + |\rho|_{H^1(T)}^2 + h_{T}^2 |\rho|_{H^2(T)}^2 \right), \qquad S\subset T,
	\]
	we get
	\begin{align*}
		&\left|\mathring{\sum_{S}}\int_S \sum_{F \supset S} \llbracket \rho(n_{\gt},\nu_{\gt}) \rrbracket_F \Tr_{\gt}(\sigma|_S) \omega_S(\gt)\right|\leq C\max_T\|g_h-g\|_{L^{\infty}(T)} h_{T}^{-2}\nrmt{\rho}_2\nrmt{\sigma}_2.
	\end{align*}
	
\end{proof}

\noindent
Collecting our results, we can state a bound on the bilinear form $B_h(\gt;\cdot,\cdot)$.
\begin{proposition} \label{prop:bhbound}
	For every $h \le h_0$, every $t \in [0,1]$, and every $\rho \in  H^2_0 S^0_2(\Omega)$, we have (with $\sigma = g_h-g$)
	\begin{align*}
		|B_h(\gt; \sigma, \rho)| 
		&\le C \left( 1 + \max_T h_T^{-2} \|g_h-g\|_{L^\infty(T)} + \max_T h_T^{-1} |g_h-g|_{W^{1,\infty}(T)} \right)  \nrmt{g_h-g}_2 \nrmt{\rho}_2.
	\end{align*}
\end{proposition}
\begin{proof}
	Combine Lemmas~\ref{lemma:Bh_vol_bound},~\ref{lemma:Bh_bnd_bound}, and~\ref{lemma:Bh_bbnd_bound}.
\end{proof}

\subsection{Convergence of the bilinear form $A_h$}
We investigate the convergence of the bilinear form $A_h(g;\cdot,\cdot)$ defined in \eqref{Ah}. We start with the codimension-0 terms.  As before, $\rho$ denotes an arbitrary member of $H^2_0 S^0_2(\Omega)$, and $\sigma = g_h-g$.

\begin{lemma}
	\label{lemma:Ah_vol_bound}
	There holds
	\begin{align*}
		\bigg|\sum_T\int_T \big( 2\sigma : \Riem(\gt) : \rho &+ \langle \Ric(\gt), \sigma \rangle_{\gt} \Tr_{\gt} \rho + \langle \Ric(\gt), \rho \rangle_{\gt} \Tr_{\gt} \sigma  \\
		&+ R(\gt) \langle J_{\gt}\sigma, \rho \rangle_{\gt} - 2\sigma:\Ric(\gt):\rho \big) \omega_T(\gt)\bigg|\leq C \|g_h-g\|_{L^2(\Omega)}\|\rho\|_{L^2(\Omega)}.
	\end{align*}
\end{lemma}
\begin{proof}
	Since we assume that $\sup_{h>0} \max_{T \in \mathcal{T}_h^N} \|g_h\|_{W^{2,\infty}(T)} < \infty$ and $\lim_{h \to 0} \|g_h-g\|_{L^\infty(\Omega)}=0$, we see from~\eqref{gtbound} that the curvature quantities associated with $\gt$ satisfy
	\[
	\|\Riem(\gt)\|_{L^\infty(T)} \le C,\qquad \|\Ric(\gt)\|_{L^\infty(T)} \le C,\qquad \|R(\gt)\|_{L^\infty(T)} \le C
	\]
	for every $h \le h_0$, every $t \in [0,1]$, and every $T \in \mathcal{T}_h^N$.  It follows that
	\begin{align*}
		&\left| \int_T \big( 2\sigma : \Riem(\gt) : \rho + \langle \Ric(\gt), \sigma \rangle_{\gt} \Tr_{\gt} \rho + \langle \Ric(\gt), \rho \rangle_{\gt} \Tr_{\gt} \sigma  
		+ R(\gt) \langle J_{\gt}\sigma, \rho \rangle_{\gt} - 2\sigma:\Ric(\gt):\rho \big) \omega_T(\gt) \right|\\
		&\qquad\le C \|\sigma\|_{L^2(T,\gt)} \|\rho\|_{L^2(T,\gt)} \leq C \|g_h-g\|_{L^2(T)} \|\rho\|_{L^2(T)}.
	\end{align*}
	Summing over all $T \in \mathcal{T}_h^N$ completes the proof.
\end{proof}
Next, we consider the codimension-1 terms in \eqref{Ah}.
\begin{lemma}
	\label{lemma:Ah_bnd_bound}
	There holds
	\begin{align*}
		\bigg|\mathring{\sum_F}\int_F \big( -3 (\sigma|_F)& : \llbracket \overline{\sff}(\gt) \rrbracket : (\rho|_F) +  \langle \llbracket \overline{\sff}(\gt) \rrbracket, \sigma|_F \rangle_{\gt} \Tr_{\gt}(\rho|_F) + \Tr_{\gt}(\sigma|_F) \langle \llbracket \overline{\sff}(\gt) \rrbracket, \rho|_F \rangle_{\gt} \\
		&- \llbracket H(\gt) \rrbracket \langle \mathbb{S}_{F,\gt} \sigma, \rho|_F \rangle_{\gt} \big) \omega_F(\gt)\bigg|\le C\max_T \left( h_T^{-1} \|g_h-g\|_{W^{1,\infty}(T)}\right) \nrm{g_h-g}_2 \nrm{\rho}_2.
	\end{align*}
	If $g_h$ is piecewise constant, then $\|g_h-g\|_{W^{1,\infty}(T)}$ can be replaced by $\|g_h-g\|_{L^{\infty}(T)}$.
\end{lemma}
\begin{proof}
	Consider an interior $(N-1)$-simplex $F$ with adjacent elements $T_1$, $T_2$ such that $F=T_1\cap T_2$. From Lemma~\ref{lem:est_facet_terms} we have (noting that $H=\Tr\sff$)
	\[
	\| \llbracket \sff(\gt) \rrbracket \|_{L^\infty(F)} + \| \llbracket H(\gt) \rrbracket \|_{L^\infty(F)} \le  C \left( \|g_h-g\|_{W^{1,\infty}(T_1)} + \|g_h-g\|_{W^{1,\infty}(T_2)} \right).
	\]
	It follows that with Lemma~\ref{lem:est_facet_terms} and the trace inequality,
	\begin{align*}
		\bigg| \int_F (\sigma|_F) : \llbracket \overline{\sff}(\gt) \rrbracket : (\rho|_F)\,\omega_F(\gt) \bigg| &\le \|\llbracket \overline{\sff}(\gt) \rrbracket \|_{L^\infty(F,\gt)} \|\sigma|_F \|_{L^2(F,\gt)} \|\rho|_F\|_{L^2(F,\gt)} \\
		&\le C \|\llbracket \overline{\sff}(\gt) \rrbracket \|_{L^\infty(F)} \|\sigma|_F \|_{L^2(F)} \|\rho|_F\|_{L^2(F)} \\
		&\le  Ch_{T_1}^{-1}  \sum_{i=1}^2 \|g_h-g\|_{W^{1,\infty}(T_i)}  \nrm{\sigma}_{2,T_1} \nrm{\rho}_{2,T_1}.
	\end{align*}
	By the shape-regularity of $\mathcal{T}_h$, we  have as in the proof of Lemma~\ref{lemma:Bh_bnd_bound}
	\begin{align*}
		&\left|\mathring{\sum_F} \int_F (\sigma|_F) : \llbracket \overline{\sff}(\gt) \rrbracket : (\rho|_F)\,\omega_F(\gt) \right| 
		\le C\max_T \left( h_T^{-1} \|g_h-g\|_{W^{1,\infty}(T)}\right) \nrm{g_h-g}_2 \nrm{\rho}_2.
	\end{align*}
	The other terms follow analogously.
\end{proof}

Finally, we estimate the codimension-2 terms in \eqref{Ah}.
\begin{lemma}
	\label{lemma:Ah_bbnd_bound}
	There holds
	\begin{align*}
		&\left|\mathring{\sum_S}\int_S \big( 2\Theta_S(\gt) \langle \sigma|_S, \rho|_S \rangle_{\gt} - \Theta_S(\gt) \Tr_{\gt}(\sigma|_S) \Tr_{\gt}(\rho|_S) \big) \omega_S(\gt)\right|\\
		&\qquad\le C \max_T \left( h_T^{-2} \|g_h-g\|_{L^\infty(T)} \right) \nrmt{g_h-g}_2  \nrmt{\rho}_2.
	\end{align*}
\end{lemma}
\begin{proof}
	The proof follows the same lines as the proof of \cite[Lemma 4.13]{gawlik2023scalar}.
\end{proof}

Collecting our results, we can state a bound on the bilinear form $A_h(\gt;\cdot,\cdot)$.
\begin{proposition} \label{prop:ahbound}
	For every $h \le h_0$, every $t \in [0,1]$, and every $\rho \in  H^2_0 S^0_2(\Omega)$, we have (with $\sigma = g_h-g$),
	\begin{align*}
		|A_h(\gt; \sigma, \rho)| 
		&\le C \left( 1 + \max_T h_T^{-2} \|g_h-g\|_{L^\infty(T)} + \max_T h_T^{-1} |g_h-g|_{W^{1,\infty}(T)} \right) \nrmt{g_h-g}_2 \nrmt{\rho}_2.
	\end{align*}
\end{proposition}
\begin{proof}
	Combine Lemmas~\ref{lemma:Ah_vol_bound},~\ref{lemma:Ah_bnd_bound}, and~\ref{lemma:Ah_bbnd_bound}.
\end{proof}

\subsection{Putting it all together}

\begin{proof}[Proof of Theorem~\ref{thm:conv} and Corollary~\ref{cor:conv}]
	Using the integral representation~\eqref{integral_rep} of the error together with Proposition~\ref{prop:bhbound} and Proposition~\ref{prop:ahbound} yields Theorem~\ref{thm:conv}. To deduce Corollary~\ref{cor:conv}, we use the approximation property \eqref{optimalorder} of the optimal-order interpolant together with the bounds 
	\begin{align*}
		\|g_h-g\|_{L^2(\Omega)} &\le |\Omega|^{1/2-1/p} \|g_h-g\|_{L^p(\Omega)}, \\
		\left( \sum_T h_T^2 |g_h-g|_{H^1(T)}^2 \right)^{1/2} &\le |\Omega|^{1/2-1/p} \left( \sum_T h_T^p |g_h-g|_{W^{1,p}(T)}^p \right)^{1/p}, \\
		\left( \sum_T h_T^4 |g_h-g|_{H^2(T)}^2 \right)^{1/2} &\le |\Omega|^{1/2-1/p} \left( \sum_T h_T^{2p} |g_h-g|_{W^{2,p}(T)}^p \right)^{1/p},
	\end{align*}
	which hold for all $p \in [2,\infty]$ (with the obvious modifications for $p=\infty$).
\end{proof}

\begin{remark}[Non-convergence for piecewise constant Regge metrics]
	Corollary~\ref{cor:conv} states optimal convergence rates of order $r+1$ if optimal-order interpolants of order $r\geq 1$ are used.  In the lowest-order case ($r=0$), Corollary~\ref{cor:conv} makes no claims about the convergence of the distributional Einstein tensor.  This is because the upper bounds on the terms estimated in Lemma~\ref{lemma:Bh_bbnd_bound} and Lemma~\ref{lemma:Ah_bbnd_bound}---which correspond to the codimension-2 terms in $B_h$ and $A_h$---do not approach $0$ as $h \to 0$ when $r=0$.  The upper bounds on all other terms estimated above---which correspond to the codimension-0 and codimension-1 terms in $B_h$ and $A_h$---still converge linearly when $r=0$.  
	
	This is in agreement with the results in \cite{gawlik2023scalar} for the scalar curvature in dimension $N\geq 3$, where also no convergence is obtained in the lowest-order case. In the numerical experiment in Section~\ref{sec:numerical}, we observe a large pre-asymptotic regime where the total error still seems to converge linearly for $r=0$. To verify that our analysis is sharp, we will see numerically that the terms estimated in Lemma~\ref{lemma:Bh_bbnd_bound} and Lemma~\ref{lemma:Ah_bbnd_bound}, as well as their sum, fail to converge when $r=0$.  We will see that these terms are quite small in comparison to the total error (even on meshes with fairly high resolution), making the non-convergence of the total error difficult to detect.
	
	An improved convergence rate for the distributional Einstein tensor is expected if one uses the canonical Regge interpolant \cite[Chapter 2]{li2018regge} to interpolate the metric in dimension $N=3$. This superconvergent behavior has been observed for the scalar curvature in \cite{gawlik2023scalar} for $N=3$ and proved for the Gaussian curvature in \cite{gopalakrishnan2022analysis} for $N=2$. This, however, will be a topic of future research.
\end{remark}

\section{Numerical examples} \label{sec:numerical}
In this section we present numerical experiments in dimension $N=3$ to illustrate the predicted convergence rates. The examples were performed in the open source finite element library NGSolve\footnote{www.ngsolve.org} \cite{schoeberl1997netgen,schoeberl2014ngsolve}, where the Regge finite elements are available for arbitrary polynomial order. We construct an optimal-order interpolant $g_h$ of a given metric tensor $g$ as follows. On each element $T$, the local $L^2$ best-approximation $\bar{g}_h|_T$ of $g|_T$ is computed. Then the tangential-tangential degrees of freedom shared by two or more neighboring elements are averaged to obtain a globally tangential-tangential continuous interpolant $g_h$.  
In \cite[Appendix A]{gawlik2023scalar} a verification that this interpolant is an optimal-order interpolant in the sense of Remark~\ref{remark:optimalorder_general} on shape-regular, quasi-uniform triangulations is given.

To compute the $H^{-2}(\Omega)$-norm of the error $f:= (G\omega)_{\rm dist}(g_h)-(G\omega)(g)$ we make use of the fact that $\|f\|_{H^{-2}(\Omega)}$ is equivalent to $\|\rho\|_{H^2(\Omega)}$, where $\rho\in  H^2_0 S^0_2(\Omega)$ solves the (component-wise) biharmonic equation $\Delta^2 \rho = f$. This equation will be solved numerically using the (Euclidean) Hellan--Herrmann--Johnson method \cite{comodi1989hellan} for each component of $\rho$. To prevent the discretization error from spoiling the real error, we use for $\rho_h$ two polynomial orders more than for $g_h$.

We consider in dimension $N=3$ the example proposed in \cite{gawlik2023scalar} on the unit cube $\Omega=(-1,1)^3$.  The Riemannian metric tensor is induced by the embedding $(x,y,z)\mapsto (x,y,z,f(x,y,z))$, where $f(x,y,z):= \frac{1}{2}(x^2+y^2+z^2)-\frac{1}{12}(x^4+y^4+z^4)$. The scalar curvature is 
{\small\[
	R(g)(x,y,z) =\frac{18\left( (1-x^2)(1-y^2)(9+q(z)) + (1-y^2)(1-z^2)(9+q(x)) + (1-z^2)(1-x^2)(9+q(y)) \right) }{(9+q(x)+q(y)+q(z))^2},
	\]}
where $q(x) = x^2 (x^2-3)^2$, and the components of the (symmetric) Ricci tensor read
\begin{align*}
	\Ric_{xx} &= \frac{9(x^2 - 1)((y^2 + z^2 - 2)(q(x) + 9) + (z^2 - 1)q(y) + q(z)(y^2 - 1))}{(9+q(x)+q(y)+q(z))^2},\\
	\Ric_{yy} &= \frac{9(y^2 - 1)((x^2 + z^2 - 2)(q(y) + 9) + (z^2 - 1)q(x) + q(z)(x^2 - 1))}{(9+q(x)+q(y)+q(z))^2},\\
	\Ric_{zz} &= \frac{9(z^2 - 1)((y^2 + x^2 - 2)(q(z) + 9) + (x^2 - 1)q(y) + q(x)(y^2 - 1))}{(9+q(x)+q(y)+q(z))^2},\\
	\Ric_{xy}&=\frac{9(y^2-3)y(x^2-3)x(x^2-1)(y^2-1)}{(9+q(x)+q(y)+q(z))^2},\\
	\Ric_{xz}&=\frac{9(z^2-3)z(x^2-3)x(x^2-1)(z^2-1)}{(9+q(x)+q(y)+q(z))^2},\\
	\Ric_{yz}&=\frac{9(y^2-3)y(z^2-3)z(z^2-1)(y^2-1)}{(9+q(x)+q(y)+q(z))^2}.
\end{align*}
The exact Einstein tensor is therefore obtained from $G=\Ric-\frac{1}{2}R\,g$.

We start with a structured mesh consisting of $6\cdot 2^{3k}$ tetrahedra, with maximum edge length $\tilde{h}=\max_T h_T=\sqrt{3}\,2^{1-k}$ (and minimal edge length $2^{1-k}$) for $k=0,1,\dots$. To avoid possible superconvergence due to mesh symmetries, we perturb each component of the inner mesh vertices by a random number drawn from a uniform distribution in the range $[-\tilde{h}\,2^{-(2\cdot 3+1)/2},\tilde{h}\,2^{-(2\cdot 3+1)/2}]$. As depicted in Figure~\ref{fig:conv_plot} and listed in Table~\ref{tab:error_N3}, linear convergence is observed when $g_h$ has polynomial degree $r=0$. For $r=1$ and $r=2$, higher convergence rates are obtained as expected.  This indicates that Theorem~\ref{thm:conv} and Corollary~\ref{cor:conv} are sharp for $r\geq 1$. For $r=0$ we observe numerically linear convergence, which is better than predicted by Theorem~\ref{thm:conv}. However, further investigation suggests that the observed linear convergence for $r=0$ is pre-asymptotic. Indeed, to test if the estimates in Lemma~\ref{lemma:Bh_bbnd_bound} and Lemma~\ref{lemma:Ah_bbnd_bound}, as well as their sum, are sharp, we compute the $H^{-2}(\Omega)$-norms of the linear functionals
\begin{align}
	\label{critical_terms}
	\begin{split}
		&F_1:\rho \mapsto -\frac{1}{2}\int_0^1 \mathring{\sum_S} \int_S \sum_{F \supset S} \llbracket \rho(n_{\gt},\nu_{\gt}) \rrbracket_F \Tr_{\gt}(\sigma|_S) \omega_S(\gt)\,dt,\\
		& F_2:\rho \mapsto \frac{1}{2}\int_0^1 \mathring{\sum_S}\int_S \big( 2\Theta_S(\gt) \langle \sigma|_S, \rho|_S \rangle_{\gt} - \Theta_S(\gt) \Tr_{\gt}(\sigma|_S) \Tr_{\gt}(\rho|_S) \big) \omega_S(\gt)\,dt,\\
		& F_3=F_1+F_2.
	\end{split} 
\end{align}
We approximate the parameter integrals above by a Gaussian quadrature with 5 and 7 Gauss points. As depicted in Figure~\ref{fig:conv_plot_crit_term}, the norms of these functionals with $r=0$ stagnate at about $2\times 10^{-4}$ after first converging with a cubic rate. Note that this stagnation is not an artifact of our use of Gaussian quadrature to approximate the parameter integral; increasing the quadrature order has a negligible effect on the results. The number $2\times 10^{-4}$ is below the overall error of about $2\times 10^{-3}$ for the finest grid; cf. Table~\ref{tab:error_N3}. Therefore, the lack of convergence predicted by Theorem~\ref{thm:conv} is not yet visible in Figure~\ref{fig:conv_plot}. For $r=1$ the expected rate of $O(h^{2})$ for \eqref{critical_terms} is clearly obtained. 

\begin{figure}
	\centering
	\resizebox{0.49\textwidth}{!}{
		\begin{tikzpicture}
			\begin{loglogaxis}[
				legend style={at={(0,0)}, anchor=south west},
				xlabel={ndof},
				ylabel={error},
				ymajorgrids=true,
				grid style=dotted,
				]
				\addlegendentry{$r=0$}
				\addplot[color=red, mark=*] coordinates {
					( 19 ,  0.1735945878157065 )
					( 98 ,  0.1096780651415121 )
					( 604 ,  0.0412060471652656 )
					( 4184 ,  0.0173449045960033 )
					( 31024 ,  0.0079745131306056 )
					( 238688 ,  0.0038371474279442 )
					( 1872064 ,  0.0018741068783692 )
				};
				\addlegendentry{$r=1$}
				\addplot[color=blue, mark=square] coordinates {
					( 92 ,  0.1041118399847353 )
					( 556 ,  0.0280865892840945 )
					( 3800 ,  0.0122388012918726 )
					( 27952 ,  0.003765248852563 )
					( 214112 ,  0.0010181427529596 )
					( 1675456 ,  0.0002654317466455 )
					( 13255040 ,  6.768895327845512e-05 )
				};
				
				\addlegendentry{$r=2$}
				\addplot[color=teal, mark=x] coordinates {
					( 255 ,  0.0136422166900814 )
					( 1662 ,  0.018982534947297 )
					( 11892 ,  0.0024204839459068 )
					( 89736 ,  0.000333604985488 )
					( 696720 ,  4.389530389609079e-05 )
					( 5489952 ,  5.630818124097331e-06 )
				};
				
				\addplot[color=black, mark=none, style=dashed] coordinates {
					( 142, {0.5*142^(-1/3)} )
					( 1912096, {0.5*1912096^(-1/3)} )
				};
				
				\addplot[color=black, mark=none, style=dashed] coordinates {
					( 142, {2*142^(-2/3)} )
					( 1912096, {2*1912096^(-2/3)} )
				};
				\addplot[color=black, mark=none, style=dashed] coordinates {
					( 142, {10*142^(-3/3)} )
					( 1912096, {10*1912096^(-3/3)} )
				};
				
			\end{loglogaxis}
			
			\node (A) at (5, 4.) [] {$O(h)$};
			\node (B) at (5.8, 1.4) [] {$O(h^2)$};
			\node (C) at (3.5, 1.9) [] {$O(h^3)$};
	\end{tikzpicture}}
	
	\caption{Convergence of the distributional Einstein tensor in the $H^{-2}(\Omega)$-norm in dimension $N=3$ with respect to the number of degrees of freedom (ndof) of $g_h$ for $r=0,1,2$.}
	\label{fig:conv_plot}
\end{figure}
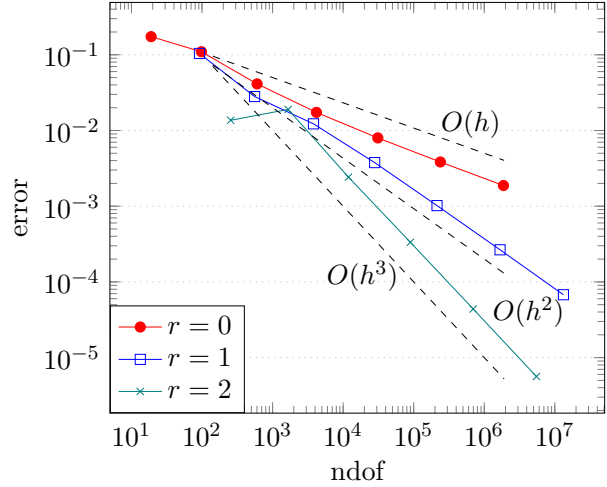

\begin{table}
	\centering
	\begin{tabular}{cccc}
		& $r=0$ & $r=1$ & $r=2$\\
		\hline
		$h$ & \begin{tabular}{@{}ll@{}}
			Error &\hspace{0.2in} Order \\
		\end{tabular} & \begin{tabular}{@{}ll@{}}
			Error &\hspace{0.2in} Order \\
		\end{tabular} & \begin{tabular}{@{}ll@{}}
			Error &\hspace{0.2in} Order \\
		\end{tabular} \\
		\hline
		\begin{tabular}{@{}l@{}}
			$3.464\cdot 10^{-0}$\\
			$1.851\cdot 10^{-0}$\\
			$9.772\cdot 10^{-1}$\\
			$5.245\cdot 10^{-1}$\\
			$2.682\cdot 10^{-1}$\\
			$1.362\cdot 10^{-1}$\\
			$6.913\cdot 10^{-2}$
		\end{tabular} & \begin{tabular}{@{}ll@{}}
			$1.736\cdot 10^{-1}$ &  \\
			$1.097\cdot 10^{-1}$ & 0.73\\
			$4.121\cdot 10^{-2}$ & 1.53\\
			$1.734\cdot 10^{-2}$ & 1.39\\
			$7.975\cdot 10^{-3}$ & 1.16\\
			$3.837\cdot 10^{-3}$ & 1.08\\
			$1.874\cdot 10^{-3}$ & 1.06
		\end{tabular} &\begin{tabular}{@{}ll@{}}
			$1.041\cdot 10^{-1}$ &  \\
			$2.809\cdot 10^{-2}$ & 2.09\\
			$1.224\cdot 10^{-2}$ & 1.3\\
			$3.765\cdot 10^{-3}$ & 1.89\\
			$1.018\cdot 10^{-3}$ & 1.95\\
			$2.654\cdot 10^{-4}$ & 1.98\\
			$6.769\cdot 10^{-5}$ & 2.01
		\end{tabular}&\begin{tabular}{@{}ll@{}}
			$1.364\cdot 10^{-2}$ &  \\
			$1.898\cdot 10^{-2}$ & -0.53\\
			$2.420\cdot 10^{-3}$ & 3.22\\
			$3.336\cdot 10^{-4}$ & 3.19\\
			$4.390\cdot 10^{-5}$ & 3.02\\
			$5.631\cdot 10^{-6}$ & 3.03\\
			\textcolor{white}{$10^ {-6}$}&
		\end{tabular}
	\end{tabular}
	\caption{Same as Figure~\ref{fig:conv_plot}, but in tabular form.}
	\label{tab:error_N3}
\end{table}

\begin{figure}
	\centering
	\resizebox{0.49\textwidth}{!}{
		\begin{tikzpicture}
			\begin{loglogaxis}[
				legend style={at={(1,1)}, anchor=north east},
				xlabel={ndof},
				ylabel={$H^{-2}(\Omega)$-norm of \eqref{critical_terms}},
				ymajorgrids=true,
				grid style=dotted,
				]
				\addlegendentry{$F_1$  gp$=5$}
				\addplot[color=red, mark=*] coordinates {
					( 19,0.02250307469867211 )
					( 98,0.00485378183252205 )
					( 604,0.0007208870687471094 )
					( 4184,0.0002478663426149099 )
					( 31024,0.000177969529542156 )
					( 238688,0.00019288263546361767 )
				};
				\addlegendentry{$F_2$ gp$=5$}
				\addplot[color=blue, mark=square] coordinates {
					( 19,0.014020337125753537 )
					( 98,0.0035638324421022054 )
					( 604,0.0005478656778400742 )
					( 4184,0.00016031839765653347 )
					( 31024,0.00010249260716239186 )
					( 238688,0.00010674987941857628 )
				};
				
				\addlegendentry{$F_3$  gp$=5$}
				\addplot[color=teal, mark=x] coordinates {
					( 19,0.026513406768010887 )
					( 98,0.00799269881705941 )
					( 604,0.0009914349697266638 )
					( 4184,0.00027995169651604906 )
					( 31024,0.0002113403680615124 )
					( 238688,0.00024387781945358033 )
				};
				
				\addlegendentry{$F_1$ gp$=7$}
				\addplot[color=black, mark=+, style=dashed] coordinates {
					( 19 ,  0.0225038134082313 )
					( 98 ,  0.0050902287757901 )
					( 604 ,  0.0006907565684082 )
					( 4184 ,  0.0002116050369151 )
					( 31024 ,  0.0001941790127142 )
					( 238688 ,  0.0002008034552744 )
				};
				\addlegendentry{$F_2$ gp$=7$}
				\addplot[color=orange, mark=x, style=dashed] coordinates {
					( 19 ,  0.0140269965877983 )
					( 98 ,  0.0036359536731315 )
					( 604 ,  0.0005386573344862 )
					( 4184 ,  0.0001447591440296 )
					( 31024 ,  0.0001119037363187 )
					( 238688 ,  0.0001103572752358 )
				};
				
				\addlegendentry{$F_3$ gp$=7$}
				\addplot[color=magenta, mark=diamond*, style=dashed] coordinates {
					( 19 ,  0.0265175559040985 )
					( 98 ,  0.0084018297052852 )
					( 604 ,  0.0009507762840758 )
					( 4184 ,  0.0002392009420786 )
					( 31024 ,  0.0002315831021922 )
					( 238688 ,  0.0002529509118983 )
				};
				
				\addplot[color=black, mark=none, style=dashed] coordinates {
					( 142, {5*142^(-3/3)} )
					( 10000, {5*10000^(-3/3)} )
				};
				
			\end{loglogaxis}
			
			\node (B) at (3.7, 3.6) [] {$O(h^3)$};
	\end{tikzpicture}}
	\resizebox{0.49\textwidth}{!}{
		\begin{tikzpicture}
			\begin{loglogaxis}[
				legend style={at={(1,1)}, anchor=north east},
				xlabel={ndof},
				ylabel={$H^{-2}(\Omega)$-norm of \eqref{critical_terms}},
				ymajorgrids=true,
				grid style=dotted,
				]
				\addlegendentry{$F_1$  gp$=5$}
				\addplot[color=red, mark=*] coordinates {
					( 92 ,  0.0107118717114141 )
					( 556 ,  0.0006236611448646 )
					( 3800 ,  0.0003384644458201 )
					( 27952 ,  8.365797475814737e-05 )
					( 214112 ,  2.263127057607625e-05 )
					( 1675456 ,  5.7328861394617e-06 )
				};
				\addlegendentry{$F_2$ gp$=5$}
				\addplot[color=blue, mark=square] coordinates {
					( 92 ,  0.0043176219697427 )
					( 556 ,  0.000304191007967 )
					( 3800 ,  0.0001757026353436 )
					( 27952 ,  4.379193612448932e-05 )
					( 214112 ,  1.1960304433469367e-05 )
					( 1675456 ,  3.0329230125491525e-06 )
				};
				
				\addlegendentry{$F_3$  gp$=5$}
				\addplot[color=teal, mark=x] coordinates {
					( 92 ,  0.0115334862424889 )
					( 556 ,  0.0008624946202095 )
					( 3800 ,  0.0004946302725587 )
					( 27952 ,  0.0001256411526558 )
					( 214112 ,  3.421453184914563e-05 )
					( 1675456 ,  8.672149822876838e-06 )
				};
				
				\addlegendentry{$F_1$ gp$=7$}
				\addplot[color=black, mark=+, style=dashed] coordinates {
					( 92 ,  0.0107108194695793 )
					( 556 ,  0.0006116495353578 )
					( 3800 ,  0.0002906420821927 )
					( 27952 ,  8.616800530248685e-05 )
					( 214112 ,  2.2475371218028536e-05 )
					( 1675456 ,  5.731940567240984e-06 )
					
				};
				\addlegendentry{$F_2$ gp$=7$}
				\addplot[color=orange, mark=x, style=dashed] coordinates {
					( 92 ,  0.004317453351147 )
					( 556 ,  0.0002990667494137 )
					( 3800 ,  0.0001515918632543 )
					( 27952 ,  4.5279217367080544e-05 )
					( 214112 ,  1.1871294447656983e-05 )
					( 1675456 ,  3.032488221989472e-06 )
				};
				
				\addlegendentry{$F_3$ gp$=7$}
				\addplot[color=magenta, mark=diamond*, style=dashed] coordinates {
					( 92 ,  0.0115322264672365 )
					( 556 ,  0.0008454877280855 )
					( 3800 ,  0.0004263860224318 )
					( 27952 ,  0.0001294954207459 )
					( 214112 ,  3.397108336984613e-05 )
					( 1675456 ,  8.670543034263645e-06 )
				};

				\addplot[color=black, mark=none, style=dashed] coordinates {
					( 6000, {0.015*6000^(-2/3)} )
					( 700000, {0.015*700000^(-2/3)} )
				};
				
				\addplot[color=black, mark=none, style=dashed] coordinates {
					( 120, {80*120^(-5/3)} )
					( 900, {80*900^(-5/3)} )
				};
				
			\end{loglogaxis}
			
			\node (A) at (4.2, 0.8) [] {$O(h^2)$};
			\node (B) at (2.2, 4.5) [] {$O(h^5)$};
	\end{tikzpicture}}
	
	\caption{Convergence of the three functionals in \eqref{critical_terms} in the $H^{-2}(\Omega)$-norm with respect to number of degrees of freedom (ndof) for 5 and 7 Gauss points (gp) in dimension $N=3$. Left: $r=0$. Right: $r=1$.}
	\label{fig:conv_plot_crit_term}
\end{figure}
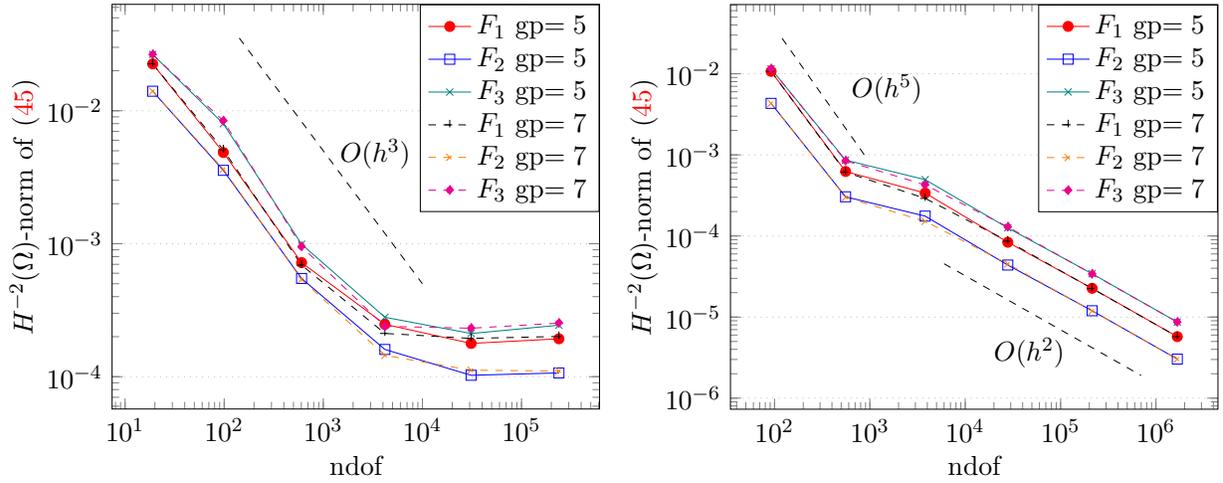

\section*{Acknowledgments}
EG was supported by NSF grant DMS-2012427 and the Simons Foundation award MP-TSM-00002615. MN acknowledges support by the Austrian Science Fund (FWF) project F\,65.

\appendix

\section{Proof of Lemma~\ref{lemma:trAhBh}}
\label{app:proof_lemma_trAhBh}

\begin{proof}[Proof of Lemma~\ref{lemma:trAhBh} (Equations~\eqref{trBh}-\eqref{trAh})]
	We will first show that $A_h(g; \sigma, vg) =  \left( \frac{N-4}{2} \right) a_h(g; \sigma, v)$ by choosing $\rho=vg$ in~(\ref{Ah}).  On each $N$-simplex $T$, we use the identities $2\sigma : \Riem : g = -2\langle \Ric, \sigma \rangle$, $\Tr(g)=N$, $\langle \Ric, g \rangle = R$, $\langle J\sigma, g \rangle = (1-\frac{N}{2})(\Tr\sigma)$, and $-2\sigma : \Ric : g = -2\langle \sigma, \Ric \rangle$ to compute
	\begin{align*}
		&2\sigma : \Riem : vg + \langle \Ric, \sigma \rangle \Tr(vg) + \langle \Ric, vg \rangle \Tr \sigma  + R \langle J\sigma, vg \rangle - 2\sigma:\Ric:vg \\
		&= -2\langle \Ric, \sigma \rangle v + N \langle \Ric, \sigma \rangle v + R(\Tr\sigma)v + (1-\tfrac{N}{2}) R (\Tr\sigma) v - 2 \langle \sigma, \Ric \rangle v \\
		&= (N-4)\langle \Ric, \sigma \rangle v + (2-\tfrac{N}{2}) R (\Tr \sigma) v \\
		&= (N-4) \langle G, \sigma \rangle v.
	\end{align*}
	On each $(N-1)$-simplex $F$, we use the identities $\Tr(g|_F) = N-1$, $\Tr \overline{\sff} = -(N-2)H$, and $\Tr(\mathbb{S}_F\sigma) = -(N-2)\Tr(\sigma|_F)$ to compute
	\begin{align*}
		-&3 (\sigma|_F) : \llbracket \overline{\sff} \rrbracket : (vg|_F) +  \langle \llbracket \overline{\sff} \rrbracket, \sigma|_F \rangle \Tr(vg|_F) + \Tr(\sigma|_F) \langle \llbracket \overline{\sff} \rrbracket, vg|_F \rangle - \llbracket H \rrbracket \langle \mathbb{S}_F \sigma, vg|_F \rangle\\
		&= -3 \langle \sigma|_F, \llbracket \overline{\sff} \rrbracket \rangle v + (N-1)  \langle \llbracket \overline{\sff} \rrbracket, \sigma|_F \rangle v - (N-2) \Tr(\sigma|_F) \llbracket H \rrbracket v + (N-2) \llbracket H \rrbracket \Tr(\sigma|_F) v  \\
		&= (N-4) \langle \llbracket \overline{\sff} \rrbracket, \sigma|_F \rangle v.
	\end{align*}
	Lastly, on each $(N-2)$-simplex $S$, we have
	\begin{align*}
		2 \Theta_S\langle \sigma|_S, vg|_S \rangle - \Theta_S \Tr(\sigma|_S) \Tr(vg|_S) 
		&= 2 \Theta_S \Tr(\sigma|_S) v - (N-2) \Theta_S \Tr(\sigma|_S) v= -(N-4) \Theta_S \Tr(\sigma|_S) v.
	\end{align*}
	It follows that
	\[
	2A_h(g;\sigma,vg) = (N-4) a_h(g;\sigma,v).
	\]
	By the symmetry of $A_h(g;\cdot,\cdot)$,~(\ref{trAh}) holds.
	
	Next we will show that $B_h(g; vg, \rho) = -\left( \frac{N-2}{2}\right) b_h(g; \rho,v)$ for all $\rho \in \Sigma$ and all $v \in \widetilde{V}$ by choosing $\sigma=vg$ in~(\ref{Bh}).  On each $N$-simplex $T$, we use the identities $Jg = g - \frac{N}{2}g = -\left(\frac{N-2}{2}\right)g$ and $J\nabla\nabla v = \nabla\nabla v - \frac{1}{2}g \Delta v$ to compute
	\begin{align*}
		2\ein(vg) 
		&= 2J\df\dv J (vg) - J\Delta (vg) \\
		&= -(N-2) J\df\dv(vg) - J(g\Delta v) \\
		&= -(N-2) J \nabla\nabla v + \left(\frac{N-2}{2}\right) g \Delta v \\
		&= -(N-2) \nabla \nabla v + \left(\frac{N-2}{2}\right) g \Delta v + \left(\frac{N-2}{2}\right) g \Delta v \\
		&= -(N-2)\mathbb{S}\nabla\nabla v.
	\end{align*}
	On each $(N-1)$-simplex $F$, we use the identities $\mathbb{S}_F g = -(N-2)g|_F$ and $\mathbb{S}\rho(n,n) = -\Tr(\rho|_F)$ to compute
	\begin{align*}
		\langle \llbracket \mathbb{S}_F \left( \nabla_n (vg) - 2(\nabla_F (vg))(n,\cdot) \right) \rrbracket, \rho|_F \rangle 
		&= \langle \llbracket \mathbb{S}_F \left( (\nabla_n v)g - 2(\nabla_F v) \otimes g(n,\cdot) \right) \rrbracket, \rho|_F \rangle \\
		&= \langle \llbracket \mathbb{S}_F \left( (\nabla_n v)g \right) \rrbracket, \rho|_F \rangle \\
		&= -(N-2)\llbracket \nabla_n v\rrbracket \Tr(\rho|_F) \\
		&= (N-2)\llbracket \nabla_n v\rrbracket \mathbb{S}\rho(n,n),
	\end{align*}
	where the second line follows from the fact that $g(n,\cdot)|_F=0$.  Also,
	\begin{align*}
		\langle &\llbracket vg(n,n) \overline{\sff} \rrbracket, \rho|_F \rangle + \langle vg|_F, \llbracket \sff \rrbracket \rangle \Tr(\rho|_F) - (vg|_F) : \llbracket \sff \rrbracket : (\rho|_F) \\
		&= \langle \llbracket \overline{\sff} \rrbracket, \rho|_F \rangle v + \llbracket H \rrbracket \Tr(\rho|_F) v - \langle \llbracket \sff \rrbracket, \rho|_F \rangle v \\
		&= \langle \llbracket \sff \rrbracket, \rho|_F \rangle v  - \langle \llbracket \sff \rrbracket, \rho|_F \rangle v \\
		&= 0.
	\end{align*}
	Lastly, the integrals over $S$ in~(\ref{Bh}) vanish when $\sigma=vg$ because $g(n,\nu)=0$.  It follows that
	\begin{align*}
		2B_h(g; vg, \rho) &= -(N-2) \sum_T \int_T \langle \mathbb{S} \nabla \nabla v, \rho \rangle \,\omega_T+ (N-2) \sum_{F}\int_F \llbracket \nabla_n v \rrbracket \mathbb{S}\rho(n,n) \,\omega_F \\
		&= -(N-2) \sum_T \int_T \langle \nabla \nabla v, \mathbb{S}\rho \rangle \,\omega_T + (N-2) \sum_{F}\int_F \llbracket \nabla_n v \rrbracket \mathbb{S}\rho(n,n)\, \omega_F \\
		&= -(N-2) b_h(g;\rho,v).
	\end{align*}
	This shows that the second equality in~(\ref{trBh}) holds. The first equality in~(\ref{trBh}) follows from the symmetry of $B_h(g;\cdot,\cdot)$.
\end{proof}
\section{Proof of Lemma~\ref{lem:eucl_distr_ein}}
\label{app:distr_Eucl_ein}

We prove Lemma~\ref{lem:eucl_distr_ein} by computing the distributional Euclidean linearized Einstein operator $\ein_{\rm dist}$.  This operator extends the classical linearized Einstein operator
\[
2\ein \sigma = 2 J \df\dv J \sigma - J \Delta \sigma
\]
to symmetric $(0,2)$-tensor fields $\sigma$ that are solely tangential-tangential continuous and piecewise smooth on an affine triangulation $\mathcal{T}$. 
In this section all differential operators, inner products, and geometric quantities, such as normal vectors, are understood in the Euclidean sense.  We also omit all (Euclidean) volume forms when writing integrals for notational simplicity.

\begin{proof}[Proof of Lemma~\ref{lem:eucl_distr_ein}]
	Let $\sigma$ be a tangential-tangential continuous and piecewise smooth symmetric $(0,2)$-tensor field, and let $\rho$ be a smooth symmetric $(0,2)$-tensor field with compact support. First, we use the definition of the distributional derivative and we integrate by parts on each element $T$ to obtain
	\begin{align}
		\llangle2&\ein_{\rm dist} \sigma, \, \rho\rrangle = \int_{\Omega} 2\sigma: \ein \rho = \sum_T\int_T  2J\sigma: \df\dv J \rho -  J\sigma: \Delta \rho\nonumber\\
		&=\sum_T-\int_T   2\dv J\sigma\cdot \dv J \rho -  \nabla J\sigma : \nabla \rho + \int_{\partial T}  2J\sigma(n,\cdot)\cdot \dv J \rho -  J\sigma : \nabla_n \rho\nonumber\\
		&=\sum_{T}\int_T 2\ein \sigma: \rho -\int_{\partial T} 2\dv J\sigma\cdot J \rho(n,\cdot) -  \nabla_n J\sigma: \rho+ \int_{\partial T} 2J\sigma(n,\cdot)\cdot \dv J \rho -  J\sigma: \nabla_n \rho.\label{eq:deriv_ein1}
	\end{align}
	Focusing on the second boundary integral in \eqref{eq:deriv_ein1}, we re-express the sum in terms of jumps of $\sigma$ to get
	\begin{align*}
		\sum_T\int_{\partial T}  2J\sigma(n,\cdot)\cdot \dv J \rho -  J\sigma: \nabla_n \rho &= \mathring{\sum_F}\int_{F}  2\llbracket J\sigma(n,\cdot)\rrbracket\cdot \dv J \rho -  \llbracket J\sigma\rrbracket: \nabla_n \rho\\
		&= \mathring{\sum_F}\int_{F} 2\llbracket J\sigma(n,\cdot)\rrbracket\cdot (\dv_F J \rho + \nabla_n(J\rho)(n,\cdot)) -  \llbracket J\sigma\rrbracket: \nabla_n \rho.
	\end{align*}
	We expand the terms not involving the surface divergence and use the fact that tangential-tangential jumps of $\sigma$ vanish. The first term reads
	\begin{align*}
		2\llbracket J\sigma(n,\cdot)\rrbracket \cdot\nabla_n(J\rho)(n,\cdot)  &= 2\llbracket\sigma(n,\cdot)-\frac{1}{2}(\Tr(\sigma|_F)+\sigma(n,n))n^\flat\rrbracket\cdot( (\nabla_n\rho)(n,\cdot)-\frac{1}{2}\Tr(\nabla_n\rho)n^\flat)\\
		&=2\llbracket\sigma(n,n)\rrbracket((\nabla_n\rho)(n,n)
		-\frac{1}{2}\Tr(\nabla_n\rho))+2\llbracket\sigma(n,\cdot)|_F\rrbracket\cdot(\nabla_n\rho)(n,\cdot)|_F \\&\quad -\llbracket\sigma(n,n)\rrbracket((\nabla_n\rho)(n,n)
		-\frac{1}{2}\Tr(\nabla_n\rho)) \\
		&=  \llbracket\sigma(n,n)\rrbracket(\nabla_n\rho)(n,n) -\frac{1}{2}\llbracket\sigma(n,n)\rrbracket\Tr( \nabla_n \rho) + 2\llbracket \sigma(n,\cdot)|_F\rrbracket\cdot(\nabla_n\rho)(n,\cdot)|_F
	\end{align*}
	and the second
	\begin{align*}
		\llbracket J\sigma\rrbracket: \nabla_n \rho
		&=  \llbracket \sigma|_F \rrbracket : (\nabla_n \rho)|_F+ 2 \llbracket \sigma(n,\cdot)|_F \rrbracket \cdot (\nabla_n \rho)(n,\cdot)|_F + \llbracket \sigma(n,n) \rrbracket (\nabla_n\rho)(n,n) \\ &\quad -\frac{1}{2}\llbracket (\Tr(\sigma|_F)+\sigma(n,n))g\rrbracket: \nabla_n \rho\\
		&= 2\llbracket \sigma(n,\cdot)|_F\rrbracket\cdot(\nabla_n\rho)(n,\cdot)|_F + \llbracket\sigma(n,n)\rrbracket(\nabla_n\rho)(n,n) -\frac{1}{2}\llbracket\sigma(n,n)\rrbracket\Tr( \nabla_n \rho).
	\end{align*}
	When subtracting them they cancel, so we can integrate by parts on each $F$ to get
	\begin{align}
		&\sum_T\int_{\partial T}  2J\sigma(n,\cdot)\cdot \dv J \rho -  J\sigma: \nabla_n \rho= \mathring{\sum_F}\int_{F}  2\llbracket J\sigma(n,\cdot)\rrbracket\cdot \dv_F J \rho\nonumber\\
		&=\mathring{\sum_F}-\int_{F} 2\nabla_F\llbracket J\sigma(n,\cdot)\rrbracket: J \rho+\int_{\partial F}2\llbracket J\sigma(n,\cdot)\rrbracket\cdot J \rho(\nu,\cdot)\nonumber\\
		&=\mathring{\sum_F}\int_{F} - \llbracket2(\nabla_F\sigma)(n,\cdot) -(\nabla_F\sigma)(n,n) \otimes n^\flat\rrbracket:\rho+\llbracket(\dv_F\sigma)(n)\rrbracket\Tr\rho+\int_{\partial F}2\llbracket J\sigma(n,\cdot)\rrbracket\cdot J \rho(\nu,\cdot).\label{eq:bnd1_term}
	\end{align}
	Next, we take a look at the first boundary integral in \eqref{eq:deriv_ein1}:
	\begin{align*}
		-\sum_T\int_{\partial T}  2\dv J\sigma\cdot J \rho(n,\cdot)-  \nabla_n J\sigma: \rho = -\mathring{\sum_F} \int_F 2\llbracket\dv J\sigma\rrbracket\cdot J \rho(n,\cdot)-  \llbracket\nabla_n J\sigma\rrbracket: \rho.
	\end{align*}
	We expand the first term and use the tangential-tangential continuity of $\sigma$ to get
	\begin{align*}
		2\llbracket\dv J\sigma&\rrbracket\cdot J\rho(n,\cdot) = \llbracket2\dv_F \sigma+2(\nabla_n J\sigma)(n,\cdot)-\nabla_F\Tr(\sigma|_F)-\nabla_F\sigma(n,n)\rrbracket\cdot J\rho(n,\cdot)\\
		&= \llbracket 2(\dv_F \sigma)(n)n^\flat+2(\nabla_n J\sigma)(n,\cdot)-\nabla_F\sigma(n,n)\rrbracket\cdot J\rho(n,\cdot)\\
		&= \llbracket2(\dv_F \sigma)(n)n^\flat+2(\nabla_n J\sigma)(n,\cdot)-\nabla_F\sigma(n,n)\rrbracket\cdot\rho(n,\cdot)\\&\quad-\llbracket(\dv_F \sigma)(n)+(\nabla_n J\sigma)(n,n)\rrbracket\Tr\rho.
	\end{align*}
	Thus,
	\begin{align}
		-\sum_T\int_{\partial T}&  2\dv J\sigma\cdot J \rho(n,\cdot)-  \nabla_n J\sigma: \rho \nonumber \\
		&= -\mathring{\sum_F} \int_F \llbracket(2\dv_F \sigma)(n)n^\flat+2(\nabla_n J\sigma)(n,\cdot)-\nabla_F\sigma(n,n)\rrbracket\cdot\rho(n,\cdot)\nonumber\\
		&\quad-\llbracket(\dv_F \sigma)(n)+(\nabla_n J\sigma)(n,n)\rrbracket\Tr\rho -  \llbracket\nabla_n J\sigma\rrbracket: \rho.\label{eq:bnd2_term}
	\end{align}
	Inserting \eqref{eq:bnd1_term} and \eqref{eq:bnd2_term} into \eqref{eq:deriv_ein1} gives
	\begin{align*}
		\llangle2\ein\sigma&,\rho\rrangle =\sum_{T}\int_T 2\ein \sigma: \rho +\mathring{\sum_F} \Big(\int_F\llbracket-(2\dv_F \sigma)(n)n^\flat-2(\nabla_n J\sigma)(n,\cdot)+\nabla_F\sigma(n,n)\rrbracket\cdot\rho(n,\cdot)\\
		&\qquad+ \llbracket(\dv_F \sigma)(n)+(\nabla_n J\sigma)(n,n)\rrbracket\Tr\rho +  \llbracket\nabla_n J\sigma\rrbracket:\rho \\
		&\qquad- \llbracket2(\nabla_F\sigma)(n,\cdot) -(\nabla_F\sigma)(n,n) \otimes n^\flat\rrbracket:\rho+\llbracket(\dv_F\sigma)(n)\rrbracket\Tr\rho \\&\qquad + \int_{\partial F} 2\llbracket (J\sigma)(n,\cdot)\rrbracket\cdot (J \rho)(\nu,\cdot)\Big)\\
		&=\sum_{T}\int_T 2\ein \sigma: \rho +\mathring{\sum_F} \Big(\int_F\llbracket-2(\dv_F \sigma)(n)\rrbracket\rho(n,n)+
		\llbracket\nabla_F\sigma(n,n)-2(\nabla_n J\sigma)(n,\cdot)\rrbracket\cdot\rho(n,\cdot)\\
		&\qquad+ \llbracket\nabla_n J\sigma-2\nabla_F\sigma(n,\cdot) +\nabla_F\sigma(n,n)\otimes n^\flat\rrbracket: \rho+\big(\llbracket2(\dv_F \sigma)(n)+(\nabla_n J\sigma)(n,n)\rrbracket\Tr\rho\big) \\
		&\qquad + \int_{\partial F} 2\llbracket (J\sigma)(n,\cdot)\rrbracket\cdot (J \rho)(\nu,\cdot)\Big)\\
		&=\sum_{T}\int_T 2\ein \sigma: \rho +\mathring{\sum_F} \Big(\int_F\llbracket\mathbb{S}_F(\nabla_n \sigma - 2(\nabla_F\sigma)(n,\cdot))\rrbracket:\rho|_F + \int_{\partial F} 2\llbracket (J\sigma)(n,\cdot)\rrbracket\cdot (J \rho)(\nu,\cdot)\Big).
	\end{align*}
	Here, we used the identities
	\begin{align*}
		&\mathbb{S}_F (\nabla_n \sigma): \rho|_F = -2(\nabla_n J\sigma)(n,\cdot)\cdot\rho(n,\cdot)+  \nabla_n J\sigma: \rho+(\nabla_n J\sigma)(n,n)\Tr\rho
	\end{align*}
	and
	\begin{align*}
		-2 \mathbb{S}_F ( \nabla_F \sigma(n,\cdot) ):\rho|_F  &= -2(\dv_F \sigma)(n)\rho(n,n)+2\nabla_F\sigma(n,n) \cdot \rho(n,\cdot)\\
		&\qquad-2\nabla_F\sigma(n,\cdot):\rho
		+2(\dv_F \sigma)(n)\Tr\rho,
	\end{align*}
	which can be verified by a straightforward computation.
	
	Finally, we consider the codimension-2 terms. Let $\{\tau_i\}_{i=1}^{N-2}$ be an orthonormal basis for the tangent space to $\partial F$.  After a short calculation, we obtain
	\begin{align*}
		2 J \sigma(n,\cdot) \cdot J \rho(\nu,\cdot) 
		&= 2\sum_{i=1}^{N-2}\sigma(n,\tau_i)\rho(\nu,\tau_i) + \sigma(n,\nu)\left( -\sum_{i=1}^{N-2}\rho(\tau_i,\tau_i) + \rho(\nu,\nu) - \rho(n,n) \right) \\
		&\quad + \left( -\sum_{i=1}^{N-2}\sigma(\tau_i,\tau_i) - \sigma(\nu,\nu) + \sigma(n,n) \right) \rho(\nu,n).
	\end{align*}
	If we think of $\sigma$ and $\rho$ as matrices and $n$, $\nu$, and $\tau_i$ as column vectors, then we can write this in matrix notation as $\Tr(A\rho)$, where
	\begin{align*}
		A &= 2 \tau \tau^T \sigma n \nu^T - \tau (\nu^T \sigma n) \tau^T + \nu \nu^T \sigma n \nu^T - n n^T \sigma \nu n^T - \nu \Tr(\tau^T \sigma \tau) n^T - \nu \nu^T \sigma \nu n^T + n n^T \sigma n \nu^T \\
		&= (2\tau \tau^T + \nu \nu^T + n n^T)\sigma n \nu^T - \tau (\nu^T \sigma n) \tau^T  - (\nu \nu^T + n n^T) \sigma \nu n^T - \nu \Tr(\tau^T \sigma \tau) n^T \\
		&= (I + \tau \tau^T)\sigma n \nu^T - \tau (\nu^T \sigma n) \tau^T - (I - \tau \tau^T) \sigma \nu n^T - \nu \Tr(\tau^T \sigma \tau) n^T \\
		&= \sigma (n\nu^T - \nu n^T) + \tau \tau^T \sigma n \nu^T - \tau (\nu^T \sigma n) \tau^T + \tau \tau^T \sigma \nu n^T - \nu \Tr(\tau^T \sigma \tau) n^T
	\end{align*}
	and $\tau$ is a matrix whose columns are $\tau_1,\tau_2,\dots,\tau_{N-2}$.
	Since $\Tr(\tau \tau^T \sigma n \nu^T \rho)  = \sum_{i=1}^{N-2} \Tr(\tau_i \tau_i^T \sigma n \nu^T \rho) = \sum_{i=1}^{N-2} \Tr(\nu \tau_i^T \sigma n \tau_i^T \rho)$, we have $\Tr(A\rho) = \Tr(B\rho)$, where
	\[
	B = \sigma(n \nu^T - \nu n^T) - \sum_{i=1}^{N-2} K_i \sigma L_i,
	\]
	$K_i = \nu \tau_i^T - \tau_i \nu^T$, and $L_i = \tau_i n^T - n \tau_i^T$.  When we integrate over $\partial F$ and sum over all $F$, we can rewrite the sum as
	\[
	\mathring{\sum_S} \sum_{T \supset S}  \int_{S} \sigma((n_1\nu_1^T - \nu_1 n_1^T) - (n_0\nu_0^T - \nu_0 n_0^T)) : \rho + \sum_{i=1}^{N-2} (K_{i0} \sigma L_{i0} - K_{i1} \sigma L_{i1}) :\rho,
	\]
	where $n_0,n_1$ are appropriately oriented unit normal vectors to the two faces of $T$ containing $S$, and similarly for $\nu_0,\nu_1$, $K_{i0},K_{i1}$, and $L_{i0},L_{i1}$.  Using the skew symmetry of $K_{i0},K_{i1},L_{i0},L_{i1}$ and the fact that $(\nu_1,n_1)$ is a rotation of $(\nu_0,n_0)$ in the 2-dimensional plane orthogonal to $S$, we can argue as in \cite[Remark 1]{christiansen2011linearization} that $K_{i0} \sigma L_{i0} - K_{i1} \sigma L_{i1}$ is symmetric and $(n_1\nu_1^T - \nu_1 n_1^T) - (n_0\nu_0^T - \nu_0 n_0^T)$ vanishes.  The tangential-tangential continuity of $\sigma$ implies that $\sum_{T \supset S} \left( \sum_{i=1}^{N-2} K_{i0} \sigma L_{i0} - K_{i1} \sigma L_{i1} \right)$ has the form $\sum_{i=1}^{N-2} a_i \tau_i^T$ for some vectors $a_1,a_2,\dots,a_{N-2}$.  Combining these facts, we deduce that
	\[
	\sum_{F \supset S}\int_S 2 \llbracket J \sigma(n,\cdot)\rrbracket \cdot J \rho(\nu,\cdot) = -\int_S \sigma_S \sum_{i=1}^{N-2} \tau_i \tau_i^T : \rho, \qquad\sigma_S := \sum_{F \supset S}\nu^T\llbracket\sigma\rrbracket_F n.
	\]
	It follows that
	\begin{align*}
		\mathring{\sum_F}\int_{\partial F} 2\llbracket J\sigma(n,\cdot)\rrbracket\cdot J \rho(\nu,\cdot)=\sum_T \sum_{F \subset \partial T}   \int_{\partial F} 2 J \sigma(n,\cdot) \cdot J \rho(\nu,\cdot) = -\mathring{\sum_S} \int_S \sum_{F \supset S} \llbracket \sigma(n,\nu) \rrbracket_F \Tr(\rho|_S).
	\end{align*}
	All together, we have
	\begin{align*}
		\llangle2\ein_{\rm dist}\sigma,\rho\rrangle = \sum_{T}\int_T 2\ein \sigma: \rho +\mathring{\sum_F} \int_F\llbracket\mathbb{S}_F(\nabla_n \sigma - 2(\nabla_F\sigma)(n,\cdot))\rrbracket:\rho|_F - \mathring{\sum_S} \int_S \sum_{F \supset S} \llbracket \sigma(n,\nu) \rrbracket_F \Tr(\rho|_S).
	\end{align*}
	Comparing the right-hand side with \eqref{eq:distr_eucl_ein} completes the proof of Lemma~\ref{lem:eucl_distr_ein}.
	
\end{proof}

\printbibliography

\end{document}